\documentclass[a4paper,11pt]{article}
\usepackage[english]{babel}
\usepackage[utf8]{inputenc}
\usepackage{amsthm}
\usepackage{amsmath}
\usepackage{amssymb}
\usepackage{mathrsfs}
\usepackage{textcomp}
\usepackage{enumerate}
\usepackage{setspace}
\usepackage{latexsym}
\usepackage{graphicx}
\usepackage{mathtools}
\usepackage{tabularray}
\usepackage{geometry}
\usepackage{float}
\usepackage{emptypage}
\usepackage{commath}
\usepackage{multicol}
\usepackage{titlesec}
\usepackage{verbatim}
\usepackage{cases}
\usepackage{orcidlink}
\usepackage{amsthm}
\newtheorem{theorem}{Theorem}[section]
\newtheorem{corollary}[theorem]{Corollary}
\newtheorem{lemma}[theorem]{Lemma}
\newtheorem{definition}[theorem]{Definition}
\newtheorem{proposition}[theorem]{Proposition}
\theoremstyle{definition}
\newtheorem{remark}[theorem]{Remark}

\numberwithin{equation}{section}
\newcommand{\e}{\varepsilon}
\renewcommand{\u}{\mathbf{u}}
\newcommand{\R}{\mathbb{R}}
\newcommand{\U}{\mathbf{U}}
\newcommand{\V}{\mathcal{V}}

\begin{document}
	
	\begin{center}
		
		{\LARGE \bf Well-Posedness and Discrete-to-Continuum Convergence of the Kirchhoff Network Model}
		
		\vskip0.5cm
		
		{\large\textsc{Joaqu\'in Oyarz\'un\orcidlink{0009-0006-1217-0615}$^1$}} \\
		{\normalsize e-mail: \texttt{joaquin.oyarzun01@universitadipavia.it}} \\
		\vskip0.35cm
		
		{\large\textsc{Marco Veneroni\orcidlink{0000-0002-8526-3154}$^{1}$}} \\
		{\normalsize e-mail: \texttt{marco.veneroni@unipv.it}} \\
		\vskip0.35cm
		
		{\footnotesize $^1$Department of Mathematics ``F. Casorati'', University of Pavia, 27100 Pavia, Italy}
		
		\vskip0.5cm
		
	\end{center}

	\begin{abstract}
		\noindent In this work, we study the connection between the bidomain model of cardiac electrophysiology and the Kirchhoff Network model (KNM), that was recently introduced by J\"ager and Tveito. In the KNM, each cell of the cardiac tissue is represented as a node in a discrete network and the dynamics of the electric potentials in the heart are described by a system of ordinary differential equations on the nodes. We first study the well-posedness of the KNM and then prove that, when the nodes are situated on the vertices of a regular lattice and the mesh size $\e$ tends to zero, the solutions of the KNM system at cell-size $\e$ converge to the solution of the bidomain model, along appropriate subsequences. 
		Regarding the convergence, we adopt the framework presented by Pennacchio, Savar\'e, and Colli Franzone in \emph{Multiscale modeling for the electrical activity of the heart}, 2005, which is based on (a) the time-discretization of the evolution equations of the bidomain and $\e$-KNM systems via a semi-implicit Euler scheme, (b) the identification of the time-discrete solutions as extremal points of a Minimizing Movement scheme, and (c) the $\Gamma$-convergence of the Minimizing Movements functionals.
	\end{abstract}

\section{Introduction}
\label{sec:intro}
The Bidomain model (BD), originally introduced by Tung in 1978 \cite{LT78}, stands out as the most comprehensive continuum-based approach for cardiac electrophysiology \cite{CFPS14}. In the model, the ordinary differential equations that describe the dynamics of the ionic current flow through the cellular membrane are coupled with the partial differential equations for the intracellular and extracellular potentials on a common domain which represents a superposition of the intra- and extracellular media. 

From a computational perspective, the BD model is widely used for simulating cardiac electrical activity (here, the related literature review would cover several volumes. We refer, as an example, to a not-so-recent, but very informative survey: \cite{CFPS14}). However, some detail is lost, as the discrete cellular structure of the myocardium is represented by a continuum domain. At the microscopic scale, the Extracellular-Membrane-Intracellular (EMI) model \cite{JaegerTveito_EMI}, which is based on the original microscopic model studied in \cite{CFS02}, provides a more accurate description of the tissue geometry, by explicitly representing individual cells. While biophysically exhaustive, the high computational cost associated with the EMI framework remains prohibitive for large-scale tissue simulations \cite{JaegerTveito_EMI, JT23}. The Kirchhoff Network Model (KNM), introduced by J\ae ger and Tveito \cite{JaegerTveito_KNM, JaegerTveito_SKNM}, is presented as an intermediate-resolution description between the EMI and BD scales. It represents cardiac tissue as a discrete resistive network, and it is characterized by CPU times $\sim$90 times shorter than the BD model on comparable instances \cite{JaegerTveito_KNM}.

Despite its success in numerical experiments, the KNM has lacked a formal mathematical foundation. Previous studies \cite{JaegerTveito_KNM, JT23} rely on numerical simulations, leaving a gap regarding the model's analytical properties and its connection to the BD model. The primary objective of this work is to analyze the KNM's well-posedness and its asymptotic relationship with the macroscopic BD model. More precisely, we claim that, from a mathematical standpoint, the KNM can be viewed as a (physiologically sound and computationally efficient) finite-differences approximation of the BD model, and, as such, it converges to the BD as the space discretization step vanishes.

To achieve these objectives, we follow the variational approach by Pennacchio, Savar\'e, and Colli Franzone \cite{PSCF05}, which originally proved the convergence of the microscopic cellular model to the macroscopic Bidomain model. In this work, we adapt their methodology to the discrete setting of the KNM. The convergence analysis relies on two steps: First, we discretize in time the KNM and the BD model; adopting the Minimizing Movements scheme introduced by De Giorgi \cite{AGS05, DG93}, the discrete-time evolution problems are reduced to the step-by-step minimization of a sequence of functionals. Then, we use $\Gamma$-convergence (see, e.g., \cite{DM93}), to show that minimizers of the KNM functionals converge to minimizers of the BD functionals, as the lattice scale $\e$ vanishes. This process incorporates the scaling problem and the derivation of the macroscopic conductivity tensors as limits of the KNM graph weights. We point out that all the results that we use regarding the BD model were obtained in \cite{CFS02,PSCF05}, in our work we just show that the finite dimensional KNM fits in the framework and then we apply the powerful machinery of \cite{PSCF05}.

\subsection{The Bidomain Model}
Let $T>0$ and $\Omega \subset \mathbb{R}^d$ ($d=2,3$) be a bounded Lipschitz domain. The state of the system is governed by the intracellular and extracellular electrical potentials, denoted by $u_i$ and $u_e$ respectively, and by a recovery variable $s$ that captures the repolarization dynamics. The fundamental variable of interest is the transmembrane potential, defined as $v = u_i - u_e$, whose evolution is described by the following system:
\begin{equation*}
	\text{(BD)}\quad   \left\{
	\begin{aligned}
		&\chi c_m \dfrac{\partial v}{\partial t} - \nabla \cdot (M_i \nabla u_i) + \chi I_{\text{ion}}(v,s) = I_{\text{app}} && \text{in } \Omega \times (0,T), \\
		&\chi c_m \dfrac{\partial v}{\partial t} + \nabla \cdot (M_e \nabla u_e) + \chi I_{\text{ion}}(v,s) = I_{\text{app}} && \text{in } \Omega \times (0,T), \\
		&\dfrac{\partial s}{\partial t} = H(v,s) && \text{in } \Omega \times (0,T),
	\end{aligned}
	\right.
\end{equation*}
subject to the following no-flux boundary and initial conditions:
\begin{equation*}
	\begin{aligned}
		M_i \nabla u_i \cdot \mathbf{n} = 0, \qquad & M_e \nabla u_e \cdot \mathbf{n} = 0 && \text{on } \partial \Omega \times (0,T), \\
		v(x, 0) = v_{0}(x), \qquad & s(x, 0) = s_{0}(x) && \text{in } \Omega,
	\end{aligned}
\end{equation*}
where $I_{\text{app}} : \Omega \times (0,T) \to \mathbb{R}$ accounts for applied source currents, while $\mathbf{n}$ denotes the outward unit normal vector to $\partial \Omega$. The parameters $c_m$ and $\chi$ represent the membrane capacitance and the surface-to-volume ratio, respectively, while $M_i, M_e: \overline \Omega \to \R^{3 \times 3}$ are the symmetric, positive definite, anisotropic conductivity tensors. 

Throughout the rest of this work, we assume that the applied source current vanishes, i.e., $I_\text{app} \equiv 0$. The specific membrane behaviour is determined by the functions $H(v, s)$ and $I_{\text{ion}}(v, w)$, which characterize the ion-channel kinetics and the associated recovery dynamics. A commonly used phenomenological model for cardiac excitation is the FitzHugh-Nagumo model, which offers a simplified description of the excitation-recovery process:
\begin{equation}
	\label{eq:FHN}
	I_{\text{ion}}(v, s) = F(v)+ \delta s, \qquad H(v,s) = \sigma v-\gamma s,
\end{equation}
where 
\begin{equation}
	\label{eq:assumption_F}
	\sigma, \gamma, \delta \geq 0 \text{ are given constants and }F \in \mathcal{C}^1(\R) \text{ satisfies: }\inf_{x \in \R} F'(x)> -\infty.
\end{equation}
Moreover, to ensure the uniqueness of a solution, a normalization condition is required:
\begin{equation}
	\label{eq:zeromean}
	\int_{\Omega} u_e(x, t) \, dx = 0, \quad \text{for a.e. } t \in (0,T).
\end{equation}
The mathematical properties of this system present significant challenges due to its degenerate parabolic nature. Rigorous foundations for existence and uniqueness have been established through different analytical perspectives. For example, Colli Franzone and Savar\'e \cite{CFS02} introduced a variational formulation and made use of the theory of degenerate evolution systems, while Bourgault, Coudi\`ere, and Pierre \cite{BCP09} reformulated the system as a single parabolic equation driven by a complex nonlocal ``bidomain" operator (see the discrete equivalent in Section \ref{sec:notation} below). The well-posedness of the BD model with Hodgkin-Huxley-based ionic dynamics was studied in \cite{V09}. 

\subsection{The Kirchhoff Network Model}
The geometry of the cardiac tissue in the KNM is represented by $N$ cells and by a set of electrical connections among them. Therefore, we choose to describe the model using the language of graphs. Let $\mathcal{G}=(\V,\mathcal{E})$ be a finite, connected, and undirected graph, where $\V := \{1, \dots, N\}$ is the set of nodes representing the cardiac cells and $\mathcal{E} \subset \V \times \V$ is the set of edges denoting the electrical connections between cells. The graph is equipped with two families of symmetric and strictly positive weights $G_i,G_e : \mathcal{E}\to\R_+$, which represent the intracellular and extracellular conductivities, respectively. 

Here, we briefly report the original description of the KNM \cite{JaegerTveito_KNM}. Given the evolution time-interval $(0,T)$, the physiological state variables are $u_i,u_e,s: \V \times (0,T) \to \R$, representing the intracellular potential, the extracellular potential, and the recovery variable, respectively. Moreover, $v:=  u_i - u_e$ is the transmembrane potential. For ease of notation, in the following, we denote $f^k:=f(k)$ and $g^{j,k}:=g(j,k)$, for any function $f:\V \to \R$ and $g:\mathcal{E}\to \R$ and, for all $k\in \V$, we let $N_k:=\{j \in \V: \{j,k\}\in \mathcal{E}\}$.

Applying Kirchhoff's Current Law to the intracellular and extracellular compartments of each cell $k$, and using Ohm's law for each pair of neighbouring cells \((j,k) \in \mathcal{E}\), we define the intracellular and extracellular currents
\begin{equation*}
	I_i^{j,k} := G_i^{j,k} (u_i^j - u_i^k), \qquad
	I_e^{j,k} := G_e^{j,k} (u_e^j - u_e^k),
\end{equation*}
The membrane current $I_m^k$ represents the total current flow across the cell membrane at node $k$ and is given by:
\begin{equation}
	\label{eq:I_m}
	I_m^k := A_m^k \left( C_m \frac{dv^k}{dt} + I_{\text{ion}}(v^k, s^k) \right).
\end{equation}
The membrane area at node \(k\) is \(A_m^k > 0\), and the membrane capacitance per unit area \(C_m > 0\) is assumed to be constant. Applying Kirchhoff's Current Law  to the intracellular compartment at node $k$ yields
\begin{equation}
	\label{eq:I_m2}
	I_m^k = \sum_{j \in  N_{k}} I_i^{j,k}.
\end{equation}
Similarly, current conservation in the extracellular compartment results in
\begin{equation}
	\label{eq:I_m3}
	\sum_{j \in  N_{k}} I_e^{j,k} + I_m^k = 0.
\end{equation}
Combining \eqref{eq:I_m2} and \eqref{eq:I_m3}, we obtain the total current conservation identity
\begin{equation}
	\label{eq:conservation}
	\sum_{j \in  N_{k}} I_i^{j,k} + \sum_{j \in  N_{k}} I_e^{j,k} = 0.
\end{equation}
To model tissue excitability, the KNM couples the transmembrane potential with a recovery variable $s^k$ at each node $k$. This relationship is governed by an ordinary differential equation describing the ionic kinetics:
\begin{equation}
	\label{eq:dynamics}
	\dfrac{ds^{k}}{dt} = H(v^{k}, s^{k}).
\end{equation}
Combining the membrane current \eqref{eq:I_m}, the current conservation \eqref{eq:conservation}, the recovery dynamics \eqref{eq:dynamics}, and adopting FitzHugh-Nagumo's structure \eqref{eq:FHN}, the system of equations governing the Kirchhoff Network Model is formulated as follows:
\begin{equation*}
	(\text{KNM} ) \quad \left\{
	\begin{array}{ll}
		C_{m}\dfrac{dv^{k}}{dt}=\dfrac{1}{A_{m}^{k}}\displaystyle\sum_{j \in  N_{k}}G_{i}^{j,k} (u_{i}^{j}-u_{i}^{k})-F(v^{k}) -\delta s^{k},\\
		\displaystyle\sum_{j \in  N_{k}} G_{i}^{j,k} (u_{i}^{j}-u_{i}^{k})+\displaystyle\sum_{j \in  N_{k}} G_{e}^{j,k} (u_{e}^{j}-u_{e}^{k})=0,\\
		\dfrac{ds^{k}}{dt}=\sigma v^{k} -\gamma s^{k}, 
	\end{array}
	\right.
\end{equation*}
subject to the initial conditions
$v^k(0)=v_0^k$, $s^k(0)=s_0^k$,
for all $k\in \V$. In order to simplify the notation, in the remainder of this article, we assume that the membrane surface ratio is node-independent and is taken into account in the graph weights. 

To study the spatially structured version of the KNM, we replace the abstract graph $\mathcal{G}=(\V,\mathcal{E})$ with a regular lattice embedded in a physical domain. Let $\e > 0$ be a dimensionless scaling parameter representing the characteristic lattice spacing and, for $x\in \R^d$, let
\begin{align}
	Q_\e(x):=x+  \left[-\dfrac{\e}{2},\dfrac{\e}{2}\right)^d,\qquad \V_\e&:= \left\{x\in \e\mathbb{Z} :  Q_\e(x) \subset \Omega\right\},\label{def:graph_e1}\\
	\Omega_\e:=\bigcup_{x\in \V_\e} Q_\e(x),\qquad
	\mathcal{E_{\e}}&:=\left\{ \{x_j,x_k\}\in \V_\e\times \V_\e: 
	|x_j - x_k|=\e\right\}.\label{def:graph_e2}
\end{align}
The associated undirected graph $\mathcal{G}_\e=(\V_\e,\mathcal{E_{\e}})$ (which, for small enough $\e$, is still connected) consists of edges connecting nearest neighbours. For each node $x_{k}\in \V_\e$, its neighbourhood is $N_{\e}^{k}:=\{x_{j}\in \V_\e: \{x_k,x_j\} \in \mathcal{E_{\e}}\}$. Let $X^\e$ be the set of all functions $w:\V_\e \to \R$. On $X^\e$ we consider the following rescaled inner product, norm, and zero-mean subspace:
\[
\langle w, \widehat w\rangle_\e
:= \e^d \!\!\!\!\sum_{x_k\in\V_\e} w(x_k) \widehat w(x_k), \,\, |w|^2_\e:=\langle w, w\rangle_\e,    \,\, X^{\e}_0:= \left\{ w\in X^\e : \sum_{x_k\in\V_\e} w(x_k) = 0\right\}.
\]	
and define $\mathbf{X}^\e:=X^\e \times X_0^\e \times X^\e$. The graph is endowed with two families of positive and symmetric conductances $G_i^\e,G_e^\e: \mathcal{E}_\e \to \R_+$. In order to ensure the usual uniform ellipticity for second order linear operators, we require that there exists $\alpha>0$ such that
\begin{equation}
	\label{eq:coercivity_strong}
	\alpha \leq  G_{i,e}^\e(x_j,x_k) \leq \alpha^{-1},\quad \forall  \{x_j,x_k\}\in \mathcal{E}_\e.
\end{equation}
Condition \eqref{eq:coercivity_strong} and the standard Poincar\'e inequality on graphs imply that there exist $\bar \alpha >0$ such that
\begin{equation}
	\label{eq:coercivity}
	\bar \alpha|y|^2 \leq \sum_{\{x_j,x_k\}\in \mathcal{E}_\e} G_{i,e}^\e(x_j,x_k)\left(\frac{y_j-y_k}{\e}\right)^2 \leq \bar \alpha^{-1}|y|^2,\quad \forall y \in X_0^\e,
\end{equation}
where $|y|$ is the euclidean (not rescaled) norm of $y \in \R^{\V_\e}$.

\subsection{Main results}
In order to state our results, we first need to define the quantities that naturally appear in the a priori estimates and in the variational formulation of the problem. For $u_i,u_e,w: \V_\e \to \R$, denoting $\bar u=(u_i,u_e)$, we define:
\begin{align}
	a^\e(\bar u) 		&:= \e^d\sum_{\{x_j,x_k\}\in \mathcal{E}_\e} G_i^\e(x_j,x_k)\left(\frac{u_i(x_j)-u_i(x_k)}{\e}\right)^2 \nonumber \\
	&\quad +\e^d\sum_{\{x_j,x_k\}\in \mathcal{E}_\e} G_e^\e(x_j,x_k)\left(\frac{u_e(x_j)-u_e(x_k)}{\e}\right)^2, \label{def:ae}\\
	b^\e(\bar u) 		&:= C_m |u_i-u_e|_\e^2,\\
	\phi^\e(\bar u) 	&:= \e^d\sum_{x_k\in \V_\e} \varphi(u_i(x_k)-u_e(x_k)),
\end{align}	 
where $\varphi$ is a positive, convex primitive function of $x \mapsto F(x)+\lambda_F x$ for a sufficiently large $\lambda_F \geq 0$ (see \eqref{def:lambda} and \eqref{def:varphi}). This variational framework will be detailed in Section \ref{sec:variational}. 

Since the initial datum for the potentials $u_i,u_e$ is assigned only on their difference $v=u_i-u_e$, we need to impose the following compatibility condition on the initial data. Given $v_0^\e \in X^\e$, we require that $\bar u_0^\e:=(u_{i,0}^\e,u_{e,0}^\e)\in X^\e \times X_0^\e$ satisfies:
\begin{equation}
	\label{eq:compatibility}
	\bar u_0^\e \in \text{argmin}\left\{ a^\e(\bar u): u_i-u_e = v_0^\e\right\}. 
\end{equation}

Our first result concerns the well-posedness of the KNM. Precisely, by KNM$-\e$ we denote the KNM system on a graph $\mathcal{G}_\e=(\V_\e,\mathcal{E_{\e}})$ defined in \eqref{def:graph_e1} and \eqref{def:graph_e2}, with conductances $G_{i,e}^\e$ satisfying \eqref{eq:coercivity} and parameters \eqref{eq:assumption_F}.
\begin{theorem}
	\label{th:wp-KNM}
	Let $\e>0$ and $T>0$ be fixed. Assuming \eqref{eq:assumption_F}, for any initial data $\u_0^\e:=(u_{i,0}^\e,u_{e,0}^\e,s_0^\e) \in  \mathbf{X}^\e$ satisfying \eqref{eq:compatibility}, there exists a unique solution
	\[
	\u^\e:=(u_i^\e,u_e^\e,s^\e)\in C^1([0,T]; \mathbf{X}^\e)
	\]
	of the (KNM$-\e$) system, satisfying $\u^\e(0)=\u_0^\e$ and the a priori estimates:
	\begin{equation}
		\label{apriori-estimate}
		\begin{aligned}
			\sup_{t \in [0,T]} \left( b^\e(\bar{u}^\e(t)) + a^\e(\bar{u}^\e(t)) +\phi^\e(\bar{u}^\e(t)) +  |s^\e(t)|^2_\e  \right) 
			\hspace{4cm} \\
			+ \int_0^T \left( |\partial_t v^\e(t)|_\e^2 + |\partial_t s^\e(t)|_\e^2 \right) dt 
			\leq C \left( b^\e(\bar{u}_0^\e) + \phi^\e(\bar{u}_0^\e) + a^\e(\bar{u}_0^\e) +|s_0^\e|^2_\e \right),
		\end{aligned}
	\end{equation}
	where $v^\e=u_i^\e-u_e^\e$, $\bar u^\e = (u_i^\e,u_e^\e)$, and the constant $C > 0$ is independent of $\e$. In addition, for each $t \in [0,T]$, the potentials $(u_i^\e, u_e^\e)$ attain the minimum constrained energy configuration for $a^\e$, that is: $a^\e(\bar{u}^\e(t)) = \min \left\{ a^\e({u}) : {u}_i - {u}_e = v^\e(t) \right\}.$
\end{theorem}

Since solutions of the KNM system are defined only on the nodes of the lattice $\V_\e$, in order to address the convergence to the continuous solutions of the BD model, it is natural to identify a function $w \in X^\e$ with a function that is constant on each cell of the lattice. Precisely, for any discrete function $w \in X^\e$, we define the piecewise constant interpolation operator $\Pi_\e : X^\e \to L^2(\Omega)$ as follows:
\begin{equation*}
	(\Pi_\e w)(x)
	:=
	\sum_{x_k\in \V_\e}
	w(x_k)\,\mathbf{1}_{Q_\e(x_k)}(x)=\begin{cases} 
		w(x_k), & x \in Q_\e(x_k) \,\,\text{and} \,\, x_k\in \V_\e, \\ 
		0, & x \in \Omega \setminus \Omega_\e. 
	\end{cases}
\end{equation*}
In line with this identification, we define the convergence of a sequence of vectors $\mathbf{u}^\e := (u_i^\e, u_e^\e, s^\e)\in \mathbf{X}^\e$ as the convergence of the corresponding piecewise-constant function $\Pi_\e\mathbf{u}^\e := (\Pi_\e u_i^\e, \Pi_\e u_e^\e, \Pi_\e s^\e)$ with respect to the strong topology of $L^2(\Omega)^3$.

The main result is the following.
\begin{theorem}
	\label{th:main}
	Let $(\e_j)_{j\in \mathbb{N}}$ be a sequence of positive real numbers converging to $0$. Let $\mathbf{u}_0 := (u_{i,0}, u_{e,0}, s_0) \in H^1(\Omega) \times H_\circ^1(\Omega) \times L^2(\Omega)$. Suppose that $\mathbf{u}_0^{\varepsilon_j} := (u_{i,0}^{\varepsilon_j}, u_{e,0}^{\varepsilon_j}, s_0^{\varepsilon_j}) \in \mathbf{X}^{\varepsilon_j}$ is a sequence of initial data satisfying \eqref{eq:compatibility} such that, as $j \to \infty$, 
	\[
	\Pi_{\varepsilon_j}\mathbf{u}_0^{\varepsilon_j} \to \mathbf{u}_0\quad \text{strongly in } L^2(\Omega)^3\quad \text{and}\quad
	\limsup_{j\to \infty} a^{\varepsilon_j}(\bar u_0^{\varepsilon_j}) +\phi^{\varepsilon_j}(\bar u_0^{\varepsilon_j}) < \infty. 
	\]
	Then, there exists a subsequence $(\e_{j_k})$ and two positive-definite conductivity tensors $M_i, M_e \in L^\infty(\Omega;  \R^{3 \times 3}_{sym})$ such that the sequence of discrete solutions $\mathbf{u}^{\varepsilon_{j_k}}$ of the KNM$-\varepsilon_{j_k}$ system satisfies 
	\[
	\Pi_{\varepsilon_{j_k}}\mathbf{u}^{\varepsilon_{j_k}}(t) \to \mathbf{u}(t)\quad \text{strongly in } L^2(\Omega)^3
	\] 
	for every $t \in [0,T]$, where $\mathbf{u}$ is the unique solution of the macroscopic bidomain model (BD), complemented by \eqref{eq:FHN}--\eqref{eq:zeromean}, with conductivity tensors $M_i$ and $M_e$.
\end{theorem}
The remainder of this paper is organized as follows. In Section 2, we recall the graph-theoretic notation, we define a discrete bidomain operator and prove local existence and uniqueness for the KNM as a graph-based dynamical system. Section 3 introduces the variational formulation and the a priori estimates, and provides the proof of Theorem \ref{th:wp-KNM}. In Section 4, we review the time discretization and the minimizing movements scheme of \cite{PSCF05} and state the uniform error estimates. Section 5 collects the results for the continuous Bidomain Model. In Section 6 we study the scaling limit as $\e \to 0$, the static $\Gamma$-convergence of the incremental functionals, and we combine all the intermediate steps into the proof of Theorem 2.


\section{Local well-posedness of the KNM system}
\label{sec:notation}
In this Section, we temporarily drop the assumption on the location of the nodes on the lattice $\e \mathbb{Z}^d$, and we study the local well-posedness of the KNM on an abstract graph. We thus omit the index $\e$ from all notation. In order to prove local existence and uniqueness, we adopt the discrete counterpart of the \textit{bidomain operator} introduced in \cite{BCP09}, thus reducing the two-equations degenerate system to one parabolic equation, driven by a more complex, but linear, operator.

\subsection{The Combinatorial Laplacian}
Let $\mathcal{G}=(\V,\mathcal{E})$ be a finite, connected, and undirected graph, with nodes $\V := \{1, \dots, N\}$, edges $\mathcal{E} \subset \V \times \V$, and symmetric weights $G : \mathcal{E}\to\R_+$. Let $X$ be the set of functions $u:\V \to \R$,  denote $u^k:=u(k)$ and for all $u,v\in X$ let $\langle u, w \rangle:=\sum_{k\in \V}u^kw^k$. For any \( u \in X \) and each node $ k \in \V$, the weighted combinatorial Laplacian operator \( \mathcal{L}_G: X \to X \) is defined by
\[
(\mathcal{L}_G u)^k  := \sum_{j \in \mathcal{N}_k} G^{j,k} (u^k - u^j).
\]
\noindent
\begin{remark}
	\label{rem:laplacian}
	We collect here some general properties of the graph Laplacian. We refer to \cite{CH97} for a complete study. 
	
	\begin{itemize}
		\item[(i)] The operator \(\mathcal{L}_G\) is linear, symmetric, and positive semidefinite: for all $u,w \in X$
		\begin{align}
			\langle \mathcal{L}_G u, w \rangle &= \sum_{k \in \V} \sum_{j \in \mathcal{N}_k} G^{j,k} (u^k - u^j) w^k   \\
			&=  \sum_{\{j,k\}\in \mathcal{E}} G^{j,k} (u^k - u^j) (w^k - w^j) =\langle u, \mathcal{L}_G w \rangle.\label{5}
		\end{align}
		In particular,
		\[
		\langle \mathcal{L}_G u, u \rangle 	=  \sum_{\{j,k\}\in \mathcal{E}} G^{j,k} (u^k - u^j)^2 \geq 0
		\] 
		and, since $\mathcal{G}$ is connected, $\mathcal L_Gu =0$ if and only if $u$ is constant. 
		\item[(ii)] \, Let $X_0:=\left\{ w\in X : \sum_{k \in \V}w^k=0\right\}$. Then $\mathcal{L}_G(X)=X_0$. Indeed, let $f \in X$ be such that  $f^k=1$ for all $k\in \V$. For any \(w \in X\)
		\[
		\sum_{k \in \V} (\mathcal{L}_G w)^k = \langle \mathcal{L}_G w, f \rangle =  \langle w, \mathcal{L}_G  f \rangle =0.
		\]
		\item[(iii)] \, The restriction of $\mathcal L_G$ to $X_0$ is positive definite and the map $\mathcal{L}_G|_{X_0} : X_0 \to X_0$ is a linear isomorphism and is thus invertible. 
	\end{itemize}	
	When $G=1$,  $\mathcal L_G$ can be interpreted as the discrete divergence of a discrete gradient, which makes it the natural graph analogue of the continuous Laplacian.
\end{remark}

\subsection{The bidomain operator}
By identifying the net currents at each node as the action of combinatorial Laplacians with weights $G_i,G_e$:
\begin{equation}
	\label{eq:graph_laplacians}
	(\mathcal{L}_{i}u_{i})^{k} = \sum_{j \in \mathcal{N}_{k}}G_{i}^{j,k} (u_{i}^{k}-u_{i}^{j}),\qquad(\mathcal{L}_{e}u_{e})^{k} = \sum_{j \in \mathcal{N}_{k}}G_{e}^{j,k} (u_{e}^{k}-u_{e}^{j}).
\end{equation}
we can replace the second equation in the KNM system with $\mathcal{L}_i u_i + \mathcal{L}_e u_e = 0$. Expressing $u_e$ in terms of the transmembrane potential as $u_e = u_i - v$ and
defining the total Laplacian operator $\mathcal{L}_T := \mathcal{L}_i + \mathcal{L}_e$, we finally obtain:  $\mathcal{L}_T u_i = \mathcal{L}_e v$. By Remark \ref{rem:laplacian},  the operator $\mathcal{L}_T|_{X_0} : X_0 \to X_0$ is symmetric and positive definite, therefore invertible with bounded inverse \((\mathcal{L}_T|_{X_0})^{-1}: X_0 \to X_0\).  Consequently, the composition
\(
X \xrightarrow{\mathcal{L}_e} X_0 \xrightarrow{(\mathcal{L}_T|_{X_0})^{-1}} X_0 \xrightarrow{\mathcal{L}_i} X_0
\)
is well defined and linear. We thus define:
\begin{equation}
	\label{eq:operator_L}
	\mathcal{L}:X \to X_0,\qquad  \mathcal{L} := \mathcal{L}_i (\mathcal{L}_T|_{X_0})^{-1}\mathcal{L}_e.
\end{equation}	 
By assigning initial data $(v_0,s_0)\in X \times X$, we can rewrite the Kirchhoff Network Model (KNM) described in Section \ref{sec:intro} as the following Cauchy problem: For $T>0$, we look for $(v,s)\in C^1([0,T];X \times X)$ such that
\begin{equation}
	\label{eq:KNM_reformulated}
	\quad \left\{
	\begin{array}{rl}
		C_m\dfrac{dv}{dt} &= - \mathcal{L}v - [F(v) + \delta s],\vspace{0.1cm}\\
		\dfrac{ds}{dt} &= \sigma v - \gamma s,\vspace{0.1cm}\\
		(v(0),s(0)) & = (v_0,s_0),
	\end{array}
	\right.
\end{equation}
where 
the membrane area parameter $A_m$ is absorbed in $G_i, G_e$, and we denote $F(v)^k:=F(v^k)$. 
\begin{proposition}
	\label{prop:local}
	There exists $T>0$ such that there the Cauchy Problem \eqref{eq:KNM_reformulated} has a unique solution $(v,s)\in C^1([0,T];X \times X)$.
\end{proposition}
\begin{proof} 
	Since $\mathcal{L}$ is linear and $F\in C^1$, the function $\mathcal F: X \times X \to X \times X$ defined by
	\[
	\mathcal{F}(v,s):=\left(\mathcal{L}v - F(v) - \delta s, \sigma v - \gamma s\right)
	\]
	is locally Lipschitz continuous. Thus, the Picard-Lindel\"of Theorem \cite{CL55} guarantees the existence and uniqueness of a local solution.
\end{proof}


\section{Variational Formulation and Global Well-posedness}
\label{sec:variational}

Let now $\mathcal G_\e$ be the graph on the $\e$-lattice, as in \eqref{def:graph_e1} and \eqref{def:graph_e2}, with weights $G_i^\e,G_e^\e$ as in \eqref{eq:coercivity}. Using the shorthand notation: $F(v^\e)(x_k):=F(v^\e(x_k))$, where $F$ satisfies \eqref{eq:assumption_F}, the KNM system can be expressed in terms of the combinatorial Laplacian \eqref{eq:graph_laplacians} as
\begin{equation}
	\label{eq:system_e}
	\left\{
	\begin{aligned}
		C_m \frac{d}{dt}v^\e 
		+ \mathcal L_{G_i^\e} u_i^{\e}
		&= -F(v^\e) - \delta s^\e, \\
		C_m \frac{d}{dt}v^\e 
		- \mathcal L_{G_e^\e} u_e^\e
		&= -F(v^\e) - \delta s^\e, \\
		\frac{d s^\e}{dt} &= \sigma v^\e - \gamma s^\e,
	\end{aligned}
	\right.
\end{equation}
where $u_i^{\e}$ and $u_e^{\e}$ are the intracellular and extracellular potentials, $s^\e$  is the recovery variable and $v^\e =u_i^\e-u_e^\e$ is the membrane potential. Additionally, we introduce:
\begin{align}
	\lambda_F &:= 1+\inf_{x \in \mathbb R}F'(x), \qquad f(x):=F(x)+\lambda_F x, \label{def:lambda}\\
	\varphi(x) &:=\int_0^x f(\rho)\,d\rho =\int_0^x F(\rho)\,d\rho+\frac{\lambda_F}{2}x^2.\label{def:varphi}
\end{align}
Observe that $f$ is a strictly increasing $\mathcal{C}^1$ function satisfying $f'(x)\ge1$ for all $x\in\mathbb R$. Therefore, $\varphi$ is a
strictly convex function with at least quadratic growth, thus satisfying:
\begin{equation}
	\label{eq:phi}
	\varphi(x) \geq \varphi(0)=0, \quad \varphi(x) \geq \dfrac{1}{2}|x|^2, \quad \forall \,x \in \mathbb{R}.
\end{equation}
We recall that  $X^\e := \mathbb{R}^{\V_\e} $, $X_0^\e$ is its zero-mean subspace, $ \mathbf{X}^\e:=X^\e \times X^{\e}_0 \times X^\e.$ For $\mathbf{u},\widehat{\mathbf{u}}\in \mathbf{X}^\e$, $\mathbf u:=(u_{i},u_{e},s)$, $\widehat{\mathbf u}=(\widehat u_i,\widehat u_e,\widehat s)$, we define the norm $\|\mathbf{u} \|_\e^2  := |u_i|_\e^2 + |u_{e}|_\e^2 + |s|_\e^2$ and the bilinear forms:
\begin{align*}
	\mathbf{b}^\e(\mathbf u,\widehat{\mathbf u}) 
	&:= C_m\langle u_i-u_e, \widehat{u}_i - \widehat{u}_e \rangle_\e 
	+ \langle s,\widehat s\rangle_\e,  \\
	\mathbf{a}^\e(\mathbf u,\widehat{\mathbf u}) 
	&:=   \langle \mathcal L_{G_i^\e} u_i, \widehat{u}_i \rangle_\e 
	+ \langle \mathcal L_{G_e^\e} u_e, \widehat{u}_e \rangle_\e 
	+\gamma\langle s,\widehat s\rangle_\e, \\
	\mathbf{g}^\e(\mathbf u, \widehat{\mathbf u}) 
	&:= \lambda_F \langle u_i-u_e, \widehat u_i - \widehat u_e \rangle_\e 
	- \delta \langle s, \widehat u_i - \widehat u_e \rangle_\e 
	+ \sigma \langle u_i-u_e, \widehat s \rangle_\e,
\end{align*}
together with the related quadratic forms:
\[
\mathbf{b}^\e(\mathbf u):= \mathbf{b}^\e(\mathbf u,\mathbf u), \quad \mathbf{a}^\e(\mathbf u):= \mathbf{a}^\e(\mathbf u,\mathbf u),
\]
and the nonlinear terms
\[
\mathbf{\phi}^\e(\mathbf u):=
\e^d\displaystyle\sum_{x_k\in\V_\e} \varphi(u_i(x_k)-u_e(x_k)),
\qquad 
\mathcal{F}^\e(\mathbf u, \widehat{\mathbf u}) := \langle f(u_i-u_e), \widehat u_i - \widehat u_e \rangle_\e.
\]
Taking the scalar product of the first equation in \eqref{eq:system_e} with $\widehat u_i$, the scalar product of the second equation with $\widehat u_e$, the third with $\widehat s$, and adding all terms, we are lead to the following problem: 
We look for a solution $\mathbf{u}^\e : (0,T) \to \mathbf{X}^\e$ of the variational evolution equation:
\begin{equation}
	\label{KNM-abstract}
	\left\{
	\begin{aligned}
		&\frac{d}{dt} \mathbf{b}^\e(\mathbf{u}^\e(t), \widehat{\mathbf{u}}) 
		+ \mathbf{a}^\e(\mathbf{u}^\e(t), \widehat{\mathbf{u}}) 
		+ 	\mathcal{F}^\e(\mathbf{u}^\e(t), \widehat{\mathbf{u}}) = \mathbf{g}^\e(\mathbf{u}^\e(t), \widehat{\mathbf{u}}), \\
		&\mathbf{b}^\e(\mathbf{u}^\e(0), \widehat{\mathbf{u}}) = \mathbf{b}^\e(\mathbf{u}_0^\e, \widehat{\mathbf{u}}),
	\end{aligned}
	\right.
\end{equation}
for any $\widehat{\mathbf{u}} \in \mathbf{X}^\e$ and initial datum $\mathbf{u}_0^\e := (u_{i,0}^\e, u_{e,0}^\e, s_0^\e)\in \mathbf{X}^\e$.

The terms appearing in the variational formulation \eqref{KNM-abstract} share the same structural properties as the corresponding terms in \cite{CFS02,PSCF05}; the next two Lemmas will make this statement precise.
\begin{lemma}
	\label{lemma:g}
	There exists $G>0$ (independent of $\e$) such that the linear  form $\mathbf{g}^\e$ satisfies:
	\begin{equation*}
		(\mathbf{g}^\e(\mathbf{u}, \widehat{\mathbf{u}}))^2 \leq G^2 \mathbf{b}^\e(\mathbf{u}) \mathbf{b}^\e(\widehat{\mathbf{u}}), \quad \forall\, \mathbf{u}, \widehat{\mathbf{u}} \in \mathbf{X}^\e,
	\end{equation*}
\end{lemma}
\begin{proof}
	The inequality follows directly from the definition of $\mathbf{g}^\e$ and $\mathbf{b}^\e$, via Cauchy-Schwarz inequality. The constant $G$ depends on the physical parameters ($C_m, \lambda_F, \delta, \sigma$) and is independent of $\e$. 
\end{proof}
\begin{lemma} 
	\label{lemma:coercive}
	The bilinear forms $\mathbf{a}^\e$ and $\mathbf{b}^\e$ are continuous, symmetric, and positive semidefinite. Moreover, there exists a constant $\nu > 0$ (independent of $\e$) such that:
	\begin{equation}
		\label{eq:ab_coercive}
		\mathbf{a}^\e(\mathbf{u}) + \mathbf{b}^\e(\mathbf{u}) \geq \nu \|\mathbf{u}\|^2_{\e}, \quad \forall \,\,\mathbf{u} \in \mathbf{X}^\e.
	\end{equation}
\end{lemma}
\begin{proof}  
	The first part of the statement follows directly from the definition of $\mathbf{a}^\e$ and $\mathbf{b}^\e$ and by Remark \ref{rem:laplacian}. Since $u_e \in X_0^\e$, by \eqref{eq:coercivity},
	\[	
	\mathbf{a}^\e(\mathbf{u}) + \mathbf{b}^\e(\mathbf{u}) 
	\geq \bar\alpha |u_e|_\e^2 + C_m |u_i - u_e|_\e^2 + (1+\gamma) |s|_\e^2.
	\]
	Let $C_1 := \min \{ \bar \alpha/2, C_m \}$ and observe that
	\[
	\bar\alpha |u_e|_\e^2 + C_m |u_i - u_e|_\e^2 
	\geq \frac{\bar\alpha}{2} |u_e|_\e^2 + C_1\left(|u_e|_\e^2 + |u_i - u_e|_\e^2 \right)
	\geq \frac{\bar\alpha}{2} |u_e|_\e^2 + \frac{C_1}{2} |u_i|_\e^2,
	\]
	so we can choose the coercivity constant $\nu: = \min \{ \frac{C_1}{2}, 1+\gamma \}$.
\end{proof}
For $\mathbf{u}\in \mathbf{X}^\e$ we define the total energy functional as:
\begin{equation}
	\label{eq:energy}
	E^\e(\mathbf{u}) := \dfrac{1}{2} \mathbf{a}^\e(\mathbf{u}) 
	+ \phi^\e(\mathbf{u}) \geq \mathbf{b}^\e(\mathbf{u}),
\end{equation}
where the inequality follows from \eqref{eq:phi} and \eqref{eq:ab_coercive}. 
Since $\mathbf{a}^\e$ is a continuous bilinear form and $\phi^\e$ is a convex lower semi-continuous functional, $E^\e$ is a proper, convex, and lower semi-continuous functional on $\mathbf{X}^\e$.

\begin{proof}[Proof of Theorem \ref{th:wp-KNM}]
	We first derive estimate \eqref{apriori-estimate}. By testing \eqref{KNM-abstract} with $\widehat{\mathbf{u}} = \partial_t \mathbf{u}^\e$, and exploiting the symmetry of $\mathbf{a}^\e$ and the relation $f = \varphi'$, we derive the following energy identity:
	\begin{equation}
		\label{energy-id-detail}
		\mathbf{b}^\e(\partial_t \mathbf{u}^\e) + \frac{d}{dt} \left( \frac{1}{2}\mathbf{a}^\e(\mathbf{u}^\e) + \phi^\e(\mathbf{u}^\e) \right) = \mathbf{g}^\e(\mathbf{u}^\e, \partial_t \mathbf{u}^\e).
	\end{equation}
	In view of Lemma \ref{lemma:g} and Young's inequality, the right-hand side is bounded by $\frac{G^2}{2} \mathbf{b}^\e(\mathbf{u}^\e) + \frac{1}{2} \mathbf{b}^\e(\partial_t \mathbf{u}^\e)$. Integrating \eqref{energy-id-detail} over the interval $(0, t)$ yields:
	\begin{equation}
		\label{pre-gronwall}
		\frac{1}{2}\int_0^t \mathbf{b}^\e(\partial_t \mathbf{u}^\e) dz + \frac{1}{2}\mathbf{a}^\e(\mathbf{u}^\e(t)) + \phi^\e(\mathbf{u}^\e(t)) \leq E^\e(\mathbf{u}_0^\e) + \frac{G^2}{2} \int_0^t \mathbf{b}^\e(\mathbf{u}^\e) dz.
	\end{equation}
	The uniform coercivity \eqref{eq:ab_coercive}, together with the non-negativity of $\phi^\e$, implies the estimate 
	\[
	\mathbf{b}^\e(\mathbf{u}^\e(t)) \leq \widehat{C}(E^\e(\mathbf{u}_0^\e) + \int_0^t \mathbf{b}^\e(\mathbf{u}^\e) dz.
	\]	 
	Consequently, Gr\"onwall's inequality ensures that $\sup_{t \in [0,T]} \mathbf{b}^\e(\mathbf{u}^\e(t)) \leq C E^\e(\mathbf{u}_0^\e)$. By substituting this result back into \eqref{pre-gronwall}, the individual bounds for $\mathbf{a}^\e, \phi^\e$ and the $L^2$-norm of $(\partial_t v_\e, \partial_t s^\e)$ are recovered from the non-negativity of the remaining terms on the left-hand side, thereby concluding \eqref{apriori-estimate}.
	
	In Proposition \ref{prop:local}, we stated the local existence and uniqueness of a solution of \eqref{KNM-abstract} on a maximal interval $[0, t_{\max})$. If $t_{\max} < T$, the norm $\|\mathbf{u}^\e(t)\|_\e$ would necessarily blow up as $t \to t_{\max}^-$. However, the uniform energy estimate \eqref{apriori-estimate} implies:
	\[ 
	\sup_{t \in [0, t_{\text{max}} )} \left[ \mathbf{b}^\e(\mathbf{u}^\e(t)) + \mathbf{a}^\e(\mathbf{u}^\e(t)) + \mathbf{\phi}^\e(\mathbf{u}^\e(t)) \right]\leq C\left( \mathbf{b}^\e(\mathbf{u}_0^\e) + \mathbf{\phi}^\e(\mathbf{u}_0^\e) + \mathbf{a}^\e(\mathbf{u}_0^\e) \right).
	\]
	This uniform bound prevents finite-time blow-up, allowing the local solution to be extended up to $T$ by standard ODE theory \cite{CL55}.
	
	Finally, the minimum energy property is a direct consequence of the Kirchhoff constraint $\mathcal{L}_{G_i^\e} u_i^{\e} + \mathcal{L}_{G_e^\e} u_e^{\e} = 0$, which acts as the Euler-Lagrange optimality condition for $\mathbf{a}^\e$ under the relation $u_i^{\e} - u_e^{\e} = v^\e$.
\end{proof}


\section{Time-discretization and uniform estimates}

Fix a final time $T > 0$ and let $\tau > 0$ denote the time step size. We define a uniform partition of the interval $[0, T]$ as $\mathcal{P}_\tau = \{ 0 = t_0 < t_1 < \dots < t_N = T \}$, where $t_n = n\tau$. Given that the convex functional $\mathbf{\phi}^\e$ is finite on the whole space $\mathbf{X}^\e$, its effective domain is simply $D(\mathbf{\phi}^\e) = \mathbf{X}^\e$. Our objective is to construct a time-discrete approximation $\{ \mathbf{U}_n^{\e,\tau} \}_{n=0}^N \subset \mathbf{X}^\e$ of the KNM solution $\mathbf{u}^\e = (u_i^{\e}, u_e^{\e}, s^\e)$ at each time step $t = t_n$, by employing a semi-implicit Euler scheme.
\vskip0.2cm
\noindent
\textbf{Semi-implicit Euler scheme $(\mathcal{P}^{\e,\tau})$.} 
\textit{Starting from the initial datum $\mathbf{U}_0^{\e,\tau} := \mathbf{u}_0^\e$, we seek a sequence $\{\mathbf{U}_n^{\e,\tau}\}_{n=1}^N \subset \mathbf{X}^\e$ that recursively solves the following system for each $n \in \{ 1, \dots, N \}$, given the previous state $\mathbf{U}_{n-1}^{\e,\tau}$:}
\begin{equation}
	\label{P-eps-tau}
	\mathbf{b}^\e \left( \frac{\mathbf{U}_n^{\e,\tau} - \mathbf{U}_{n-1}^{\e,\tau}}{\tau}, \widehat{\mathbf{U}} \right) + \mathbf{a}^\e(\mathbf{U}_n^{\e,\tau}, \widehat{\mathbf{U}}) + \mathcal{F}^\e(\mathbf{U}_n^{\e,\tau}, \widehat{\mathbf{U}}) = \mathbf{g}^\e(\mathbf{U}_{n-1}^{\e,\tau}, \widehat{\mathbf{U}}), \quad \forall\,\, \widehat{\mathbf{U}} \in \mathbf{X}^\e.
\end{equation}
\vskip0.2cm
\noindent
\textbf{Minimizing movement problem $(\mathcal{P}^{\e,\tau}_n)$.} \textit{For each $n \in \{ 1, \dots, N \}$, find $\mathbf{U}_n^{\e,\tau} \in \mathbf{X}^\e$ which attains the minimum $\min \{ \Psi_{n}^{\e,\tau}(\mathbf{U}) : \mathbf{U} \in \mathbf{X}^\e \}$, where:}
\begin{equation}
	\label{MM-eps}
	\Psi_{n}^{\e,\tau}(\mathbf{U}) = \frac{1}{2\tau} \mathbf{b}^\e\left(\mathbf{U} - \mathbf{U}_{n-1}^{\e,\tau}\right) + \frac{1}{2} \mathbf{a}^\e(\mathbf{U}) + \mathbf{\phi}^\e(\mathbf{U}) - \mathbf{g}^\e\left(\mathbf{U}_{n-1}^{\e,\tau}, \mathbf{U}\right).
\end{equation}
\begin{proposition}
	For a fixed $\tau > 0$ and any given $\mathbf{U}_{n-1}^{\e,\tau} \in \mathbf{X}^\e$, the incremental minimization problem $(\mathcal{P}_n^{\e,\tau})$ admits a unique minimizer $\mathbf{U}_n^{\e,\tau} \in \mathbf{X}^\e$ for each $n \in \{ 1, \dots, N \}$. Furthermore, the sequence $\{\mathbf{U}_n^{\e,\tau}\}_{n=0}^N$ generated by these minimizations satisfies the Euler scheme $(\mathcal{P}^{\e,\tau})$. Conversely, any solution to $(\mathcal{P}^{\e,\tau})$ minimizes the functional $\Psi_n^{\e,\tau}$ at each time step.
\end{proposition}
\begin{proof} 
	Fix $n \in \{ 1, \dots, N \}$ and consider the functional $\Psi_n^{\e,\tau}(\mathbf{U})$ defined in \eqref{MM-eps}. By Lemma \ref{lemma:coercive}, the quadratic form $\mathbf{a}^\e + \mathbf{b}^\e$ is coercive on $\mathbf{X}^\e$. Since $\mathbf{\phi}^\e$ is a strictly convex functional, $\Psi_n^{\e,\tau}$ is strictly convex and coercive on $\mathbf{X}^\e$. This implies the existence and uniqueness of a minimizer $\mathbf{U}_n^{\e,\tau}$. The first-order optimality condition for this minimization problem reads:
	\[
	\frac{1}{\tau} \mathbf{b}^\e(\mathbf{U}_n^{\e,\tau} - \mathbf{U}_{n-1}^{\e,\tau}, \mathbf{W}) 
	+ \mathbf{a}^\e(\mathbf{U}_n^{\e,\tau}, \mathbf{W}) 
	+ \langle \mathbf{\phi}'(\mathbf{U}_n^{\e,\tau}), \mathbf{W} \rangle - \mathbf{g}^\e(\mathbf{U}_{n-1}^{\e,\tau}, \mathbf{W}) = 0,
	\]
	for all $\mathbf{W} \in \mathbf{X}^\e$. Recognizing that $\langle \mathbf{\phi}'(\mathbf{U}), \mathbf{W} \rangle = \mathcal{F}^\e(\mathbf{U}, \mathbf{W})$, this identity coincides with the weak formulation of the Euler scheme $(\mathcal{P}^{\e,\tau})$. Conversely, any solution of the Euler scheme satisfies this Euler-Lagrange condition, which implies the minimization of $\Psi_n^{\e,\tau}$.
\end{proof}

In order to analyze the convergence of the discrete scheme as $\tau \to 0$, we extend the sequence of minimizers $\{\mathbf{U}_n^{\e,\tau}\}_{n=0}^N$ to the continuous interval $[0,T]$. Specifically, for $t \in ((n-1)\tau, n\tau]$ and each $n \in \{1, \dots, N\}$, we introduce the following time-interpolants:
\begin{itemize}
	\item The piecewise linear interpolant: $\widetilde{\mathbf{U}}^{\e, \tau}(t) := (1 - \ell(t))\mathbf{U}_{n-1}^{\e, \tau} + \ell(t)\mathbf{U}_n^{\e, \tau}$, where $ \ell(t) = \frac{t}{\tau} - (n-1)$.
	\item The right-continuous piecewise constant interpolant: $\overline{\mathbf{U}}^{\e, \tau}(t) := \mathbf{U}_n^{\e, \tau}.$
	\item The left-continuous piecewise constant interpolant: $\underline{\mathbf{U}}^{\e, \tau}(t) := \mathbf{U}_{n-1}^{\e, \tau}.$
\end{itemize}
By construction, the linear interpolant $\widetilde{\mathbf{U}}^{\e, \tau}$ is differentiable a.e. in $(0,T)$, with its time derivative given by:
\begin{equation*}\label{total-derivative-interpolants}
	\frac{d}{dt} \widetilde{\mathbf{U}}^{\e, \tau}(t) = \frac{\overline{\mathbf{U}}^{\e, \tau}(t) - \underline{\mathbf{U}}^{\e, \tau}(t)}{\tau}, \quad \text{for } t \in ((n-1)\tau, n\tau).
\end{equation*}
Since the KNM system \eqref{KNM-abstract} shares the same structural properties (Lemma \ref{lemma:g} and Lemma \ref{lemma:coercive}) as the micro and macroscopic variational formulation in \cite{PSCF05}, we can employ the $\e$-uniform estimates for the time-discrete solutions. We recall here the main estimates, referring to \cite[Section 5]{PSCF05} for all proofs.
\smallskip

\noindent \textbf{Stability estimates.} Let $\{\mathbf{U}_n^{\e,\tau}\}_{n=1}^N$ be the sequence of minimizers. Then there exists a constant $C > 0$ (independent of $\tau$ and $\e$) such that:
\begin{equation}
	\label{stability-estimate-knm}
	\sum_{n=1}^{N} \tau \mathbf{b}^\e\left(\frac{\mathbf{U}_n^{\e,\tau} - \mathbf{U}_{n-1}^{\e,\tau}}{\tau}\right) + \sum_{n=1}^{N} \mathbf{a}^\e(\mathbf{U}_n^{\e,\tau} - \mathbf{U}_{n-1}^{\e,\tau}) + \max_{0 \le n \le N} E^\e(\mathbf{U}_n^{\e,\tau}) \le C E^\e(\mathbf{u}_0^\e).
\end{equation}
\smallskip

\noindent \textbf{A continuous variational inequality.} The time interpolants $\widetilde{\mathbf{U}}^{\e, \tau}$, $\overline{\mathbf{U}}^{\e, \tau}$, and $\underline{\mathbf{U}}^{\e, \tau}$ satisfy the following uniform estimate:
\begin{equation}\label{continuous-stability-tilde}
	\begin{split}
		\int_0^T \left[ \mathbf{b}^\e\left(\frac{d}{dt} \widetilde{\mathbf{U}}^{\e, \tau}(t)\right)\right. &+\left. \frac{1}{\tau}\mathbf{a}^\e(\overline{\mathbf{U}}^{\e, \tau}(t) - \underline{\mathbf{U}}^{\e, \tau}(t)) \right] dt \\
		&+ \sup_{t \in [0,T]} E^\e(\overline{\mathbf{U}}^{\e, \tau}(t)) \le C E^\e(\mathbf{u}_0^\e).
	\end{split}
\end{equation}
\smallskip

\noindent \textbf{A Gronwall-type estimate for the error.} 
Let $(\widetilde{\mathbf{U}}^{\e, \tau}, \overline{\mathbf{U}}^{\e, \tau})$ and $(\widetilde{\mathbf{U}}^{\e, \eta}, \overline{\mathbf{U}}^{\e, \eta})$ be the interpolants associated with time steps $\tau, \eta > 0$ for a fixed parameter $\e > 0$. There exists a constant $C = C(T, G) > 0$ such that:
\begin{equation}
	\label{cauchy-estimate-knm}
	\sup_{t \in [0,T]} \mathbf{b}^\e(\widetilde{\mathbf{U}}^{\e, \tau}(t) - \widetilde{\mathbf{U}}^{\e, \eta}(t)) + \int_0^T \mathbf{a}^\e(\overline{\mathbf{U}}^{\e, \tau}(t) - \overline{\mathbf{U}}^{\e, \eta}(t)) \, dt \le C(\tau + \eta) E^\e(\mathbf{u}_0^\e).
\end{equation}

\begin{lemma}
	\label{lemma:tau_convergence}
	For fixed $\e>0$, let $\mathbf{u}^\e = (u_i^\e,u_e^\e,s^\e):(0,T)\to\V_0^\e$ be the unique solution of \eqref{KNM-abstract}. For all sequences $(\tau_k)_{k \in \mathbb{N}}$, $\tau_k\to 0$, 
	\begin{align}
		\widetilde{\mathbf{U}}^{\e, \tau_k} \to \mathbf{u}^\e\quad \text{strongly in }L^2(0,T; \V_0^\e),\label{eq:strongL2}\\
		\left(\widetilde V^{\e,\tau_k},\widetilde s^{\e,\tau_k}\right) \to \left(v^\e,s^\e\right)\quad 
		\text{strongly in }C^0([0,T]; X^\e \times X^\e), \label{eq:strongC0}\\
		\frac{d}{dt}\left(\widetilde V^{\e,\tau_k},\widetilde s^{\e,\tau_k}\right) \to \frac{d}{dt}\left(v^\e,s^\e\right)\quad 
		\text{weakly in }L^2(0,T; X^\e \times X^\e),\label{eq:weakL2}
	\end{align}	
	where $\widetilde V^{\e,\tau}:= \widetilde{U}_i^{\e, \tau}-\widetilde{U}_e^{\e, \tau}$ and $v^\e:=u_i^\e -u_e^\e$.
\end{lemma}
\begin{proof}
	Let $(\tau_k)_{k \in \mathbb{N}}$ be fixed. By definition of $\mathbf{b}^\e$ and \eqref{cauchy-estimate-knm}, $(\widetilde V^{\e,\tau_k})_k$ and $(\widetilde{s}^{\e,\tau_k})_k$ are Cauchy sequences in $C^0([0,T]; X^\e)$. Moreover, by \eqref{eq:ab_coercive} and \eqref{cauchy-estimate-knm}
	\[
	\int_0^T  \left|\overline{\mathbf{U}}^{\e, \tau}(t) - \overline{\mathbf{U}}^{\e, \eta}(t))\right|_\e^2 \, dt \le C(\tau + \eta) E^\e(\mathbf{u}_0^\e),
	\]
	so that $(\overline{\mathbf{U}}^{\e, \tau_k})_k$ is Cauchy in $L^2(0,T;\V_0^\e)$. Finally, by \eqref{continuous-stability-tilde}, $\frac{d}{dt}\widetilde V^{\e,\tau_k}$ and $\frac{d}{dt}\widetilde{s}^{\e,\tau_k}$ are bounded in $L^2(0,T;X^\e)$. Up to extracting further subsequences, not relabeled, we can find $\mathbf{u}^\e = (u_i^\e,u_e^\e,s^\e):(0,T)\to\V_0^\e$ such that \eqref{eq:strongL2}, \eqref{eq:strongC0}, and \eqref{eq:weakL2} hold. Since $(\mathbf{U}^{\e,\tau_k}_n)_n$ is a solution of $(\mathcal{P}^{\e,\tau})$, integrating the Euler Scheme \eqref{P-eps-tau} in time yields
	\begin{equation}
		\label{eq:test}
		\begin{split}
			\int_0^T \mathbf{b}^\e &\left( \frac{d}{dt} \widetilde{\mathbf{U}}^{\tau_k,\e}(t), \widehat{\mathbf{U}}(t) \right) + \mathbf{a}^\e\left(\overline{\mathbf{U}}^{\e,\tau_k}(t), \widehat{\mathbf{U}}(t)\right)dt\\
			& +\int_0^T \mathcal{F}^\e\left(\overline{\mathbf{U}}^{\e,\tau_k}(t), \widehat{\mathbf{U}}(t)\right) -\mathbf{g}^\e\left(\underline{\mathbf{U}}^{\e,\tau_k}(t), \widehat{\mathbf{U}}(t)\right)dt=0,
		\end{split}
	\end{equation}
	for all $\widehat{\mathbf{U}} \in L^2(0,T;\mathbf{X}^\e)$. Since the piecewise constant interpolants of $V_n^{\e,\tau}:= U_{i,n}^{\e, \tau}-U_{e,n}^{\e, \tau}$ converge strongly in $L^\infty(0,T;X^\e)$, we can pass to the limit in \eqref{eq:test} and obtain
	\begin{equation*}
		\begin{split}
			\int_0^T \mathbf{b}^\e \left( \frac{d}{dt} \mathbf{u}^\e (t), \widehat{\mathbf{U}}(t) \right) & + \mathbf{a}^\e\left(\mathbf{u}^\e(t), \widehat{\mathbf{U}}(t)\right)dt\\
			& +\int_0^T \mathcal{F}^\e\left(\mathbf{u}^\e, \widehat{\mathbf{U}}(t)\right) -\mathbf{g}^\e\left(\mathbf{u}^\e(t), \widehat{\mathbf{U}}(t)\right)dt=0.
		\end{split}
	\end{equation*}
	Since $\widehat{\mathbf{U}}$ is arbitrary, we conclude that $\mathbf{u}^\e$ is the solution of \eqref{KNM-abstract}, which is unique by Theorem \ref{th:wp-KNM}. Finally, since the limit is unique and independent of the choice of subsequence, a standard compactness argument ensures that the entire sequence of interpolants $\{\widetilde{\mathbf{U}}^{\e, \tau}\}_{\tau > 0}$ converges to $\mathbf{u}^\e$ as $\tau \to 0$.
\end{proof}
By Lemma \ref{lemma:tau_convergence}, passing to the limit as $\eta \to 0$ in \eqref{cauchy-estimate-knm} immediately yields the following fundamental estimate.
\begin{theorem}
	Let $\mathbf{u}^\e$ be the unique solution of the time-continuous problem \eqref{KNM-abstract} and let $\widetilde{\U}^{\e, \tau}$ be the piecewise linear interpolant of the solutions $\{\U^{\e,\tau}_n\}_{n=0}^N$ of the minimizing movement scheme \eqref{MM-eps}. Then, there exists $C > 0$ (independent of $\tau$ and $\e$) such that:
	\begin{equation}
		\label{final-error-estimate}
		\sup_{t \in [0, T]} \mathbf{b}^\e(\mathbf{u}^\e(t) - \widetilde{\mathbf{U}}^{\e, \tau}(t)) + \int_0^T \mathbf{a}^\e(\mathbf{u}^\e(t) - \overline{\mathbf{U}}^{\e, \tau}(t)) \, dt \le C \tau E^\e(\mathbf{u}_0^\e).
	\end{equation}
\end{theorem}


\section{The Bidomain Model}
In this section, we collect the main results, obtained in \cite{CFS02, PSCF05}, that characterize the continuous Bidomain model. 
We introduce the continuous state space $\mathbf{X} := H^1(\Omega) \times H_\circ^1(\Omega) \times L^2(\Omega)$, where the second component ensures the uniqueness of the extracellular potential via a zero-mean constraint:
\(
H_\circ^1(\Omega) := \left\{ w \in H^1(\Omega) : \int_{\Omega} w \, dx = 0 \right\}.
\)
For any state $\mathbf{u} = (u_i, u_e, s) \in \mathbf{X}$ and test functions $\widehat{\mathbf{u}} = (\widehat u_i, \widehat u_e, \widehat s) \in \mathbf{X}$, we define the following bilinear forms:
\begin{align*}
	\mathbf{b}(\mathbf{u}, \widehat{\mathbf{u}}) &:= C_m \int_\Omega (u_i - u_e)(\widehat{u}_i - \widehat{u}_e) \, dx + \int_\Omega s \widehat{s} \, dx, \\
	\mathbf{a}(\mathbf{u}, \widehat{\mathbf{u}}) &:= \sum_{\alpha \in \{i,e\}} \int_\Omega \nabla u_\alpha \cdot M_\alpha \nabla \widehat u_\alpha \, dx + \gamma \int_\Omega s \widehat s \, dx, \\
	\mathbf{g}(\mathbf{u}, \widehat{\mathbf{u}}) &:= \lambda_F \int_\Omega (u_i - u_e)(\widehat{u}_i - \widehat{u}_e) \, dx - \delta \int_\Omega s (\widehat{u}_i - \widehat{u}_e) \, dx + \sigma \int_\Omega (u_i - u_e) \widehat{s} \, dx.
\end{align*}
The corresponding quadratic forms are $\mathbf{b}(\mathbf u):=\mathbf{b}(\mathbf{u}, \mathbf{u})$ and $\mathbf{a}(\mathbf u):=\mathbf{a}(\mathbf{u}, \mathbf{u})$. The conductivity matrices $M_i, M_e$ are assumed to satisfy the uniform ellipticity condition:
\[
\mu |y|^2 \le y \cdot M_{i,e}(x) y \le \mu^{-1} |y|^2, \qquad \, \forall\,\,y \in \mathbb{R}^d,\quad \text{for a.e. }x\in \Omega.
\]
We further define the convex potential and the nonlinear form
\[
\phi(\mathbf{u}) := \displaystyle\int_\Omega \varphi(u_i - u_e) \, dx,\qquad 
\mathcal{F}(\mathbf{u}, \widehat{\mathbf{u}}) := \displaystyle\int_\Omega f(u_i - u_e)(\widehat{u}_i - \widehat{u}_e) \, dx.
\]
Considering $\mathbf{u}^\e$ as time-dependent functions with values in the space-dependent functional space $\mathbf{X}$, the problem can be solved by looking for a solution $\mathbf{u}: (0,T) \to \mathbf{X}$ of the abstract variational equation:
\begin{equation}
	\label{Bidomain-abstract}
	\left\{
	\begin{aligned}
		&\frac{d}{dt} \mathbf{b}(\mathbf{u}(t), \widehat{\mathbf{u}}) + \mathbf{a}(\mathbf{u}(t), \widehat{\mathbf{u}}) + \mathcal{F}(\mathbf{u}(t), \widehat{\mathbf{u}}) = \mathbf{g}(\mathbf{u}(t), \widehat{\mathbf{u}}), \\
		&\mathbf{b}(\mathbf{u}(0), \widehat{\mathbf{u}}) = \mathbf{b}(\mathbf{u}_0, \widehat{\mathbf{u}}),
	\end{aligned}
	\right.
\end{equation}
for any $\widehat{\mathbf{u}} \in \mathbf{X}$ and the initial datum $\mathbf{u}_0 := (u_{i,0}, u_{e,0}, s_0)$ is related to the initial transmembrane potential $v_0$ through the compatibility conditions:
\begin{equation}\label{initial-condition-relation1}
	u_{i,0} - u_{e,0}= v_0, \quad \mathbf{a}(\mathbf{u}_0) = \min \left\{ \mathbf{a}(\widehat{\mathbf{u}}) : \widehat{u}_i - \widehat{u}_e = v_0\right\}.
\end{equation}
Regarding the structural properties of \eqref{Bidomain-abstract}, we note that:
\begin{itemize}
	\item The form $\mathbf{b}(\cdot, \cdot)$ is symmetric and induces a semi-norm on $\mathbf{X}$. Its kernel is infinite-dimensional, reflecting the degenerate parabolic nature of the system.
	\item The form $\mathbf{a}(\cdot, \cdot)$ is continuous, symmetric, and coercive on $\mathbf{X}$ due to the ellipticity of the conductivity tensors and the zero-mean constraint on $H_\circ^1(\Omega)$.
\end{itemize}

As in Section 4, we approximate \eqref{Bidomain-abstract} via a recursive minimization procedure. For a given time step $\tau = T/N$, we define the sequence $\{\mathbf{U}_n^\tau\}_{n=1}^N \subset \mathbf{X}$ as follows:
\vskip0.3cm
\noindent
\textbf{Minimizing movement problem $(\mathcal{P}^{\tau}_n)$.} \textit{For each $n \in \{ 1,\dots,N \}$, find $\mathbf{U}_n^{\tau}\in \mathbf{X}$ which attains the minimum $\min \left\{ \Psi_{n}^{\tau}(\mathbf{U}) : \mathbf{U} \in \mathbf{X} \right\}$, where:}
\begin{equation}
	\label{MM-eps-tt}
	\Psi_{n}^{\tau}(\mathbf{U}) := \frac{1}{2\tau}\, \mathbf{b}\left(\mathbf{U} - \mathbf{U}_{n-1}^\tau\right) + \frac{1}{2}\, \mathbf{a}(\mathbf{U}) + \phi(\mathbf{U}) - \mathbf{g}\left(\mathbf{U}_{n-1}^\tau, \mathbf{U}\right).
\end{equation}
For every $n \in \{1, \dots, N\}$, the functional in \eqref{MM-eps-tt} is strictly convex and coercive on $\mathbf{X}$. Consequently, $(\mathcal{P}_n^\tau)$ admits a unique solution $\mathbf{U}_n^\tau \in \mathbf{X}$.

The global existence, stability, and convergence of this scheme are summarized in the following theorems (see \cite{PSCF05} for detailed proofs).
\begin{theorem}
	Let $\mathbf{u}_0 \in \mathbf{X}$ satisfy $\mathbf{b}(\mathbf{u}_0) + \phi(\mathbf{u}_0) + \mathbf{a}(\mathbf{u}_0) < \infty$. Then there exists a unique solution $\mathbf{u} \in C^0([0, T]; \mathbf{X})$ to \eqref{Bidomain-abstract}. Furthermore, the solution satisfies the following a priori estimate:
	\begin{equation}
		\label{eq:stability_continuous_refined}
		\begin{aligned}
			\sup_{t \in [0,T]} \Big( \mathbf{b}(\mathbf{u}(t)) + \phi(\mathbf{u}(t)) + \mathbf{a}(\mathbf{u}(t)) \Big) + \int_0^T \left( \|\partial_t v_\e\|_\e^2 + \|\partial_t s^\e\|_\e^2 \right) \,dt &\\ \le C  (\mathbf{b}(\mathbf{u}_0) + \phi(\mathbf{u}_0) + \mathbf{a}(\mathbf{u}_0)),
		\end{aligned}
	\end{equation}
	and the pair $(u_i, u_e)$ satisfies the minimum energy principle $\mathbf{a}(\mathbf{u}(t)) = \min \left\{ \mathbf{a}(\widehat{\mathbf{u}}) : \widehat{u}_i - \widehat{u}_e\right.$ $\left.= u_i(t) - u_e(t) \right\}$ for all $t \in [0,T]$.
\end{theorem}

\begin{theorem}
	Let $\mathbf{u}$ be the unique solution of \eqref{Bidomain-abstract} and let $\mathbf{U}^{\tau}$ be its piecewise linear interpolant. There exists a constant $C > 0$, independent of $\tau$, such that:
	\begin{equation}
		\label{final-error-estimate-bidomain}
		\sup_{t \in [0,T]} \mathbf{b}(\mathbf{u}(t) - \mathbf{U}^{\tau}(t)) + \int_0^T \mathbf{a}(\mathbf{u}(t) - \mathbf{U}^{\tau}(t)) \, dt \le C \tau (\mathbf{b}(\mathbf{u}_0) + \phi(\mathbf{u}_0) + \mathbf{a}(\mathbf{u}_0)).
	\end{equation}
\end{theorem}


\section{$\Gamma$-convergence of the stationary problem}
In order to connect the discrete functions defined on the lattice $\V_\e$ with functions in $L^2(\Omega)$, we identify functions on $\V_\e$ with piecewise constant functions, via the spatial interpolation operator $\Pi_\e :X^\e \to L^2(\Omega)$
\begin{equation}
	\label{def:interpol}
	(\Pi_\e w)(x)
	:=
	\sum_{x_k\in \V_\e}
	w(x_k)\,\mathbf{1}_{Q_\e(x_k)}(x)=\begin{cases} 
		w(x_k), & x \in Q_\e(x_k) \,\,\text{and} \,\, x_k\in \V_\e, \\ 
		0, & x \in \Omega \setminus \Omega_\e
	\end{cases}
\end{equation}
and we let $\Pi_\e \mathbf u: = (\Pi_\e u_i,\Pi_\e u_e,\Pi_\e s)$, for $\mathbf u\in \mathbf{X}^\e$. Recall that for  $w,\widehat w\in X^\e$
\begin{equation*}
	\langle w,\widehat w\rangle_\e
	:=\e^d\sum_{x_k\in\V_\e}w(x_k)\,\widehat w(x_k),
	\qquad
	|w|_\e^2:=\langle w,w\rangle_\e 	=
	\|\Pi_\e w\|_{L^2(\Omega)}^2.
\end{equation*}
If $w\in X_0^\e$, then
\begin{equation}
	\label{eq:zerodecomposition}
	\int_\Omega \Pi_\e w(x)\,dx
	=
	\e^d\sum_{x_k\in\V_\e} w(x_k)
	= 0.
\end{equation}

\begin{remark}
	\label{lemma:meanzero}
	Since $\Omega$ is a bounded Lipschitz domain, $\Omega \setminus \Omega_\e$ is a boundary layer with thickness of order $\e$. In particular, there exists a constant $B>0$ such that $|\Omega \setminus \Omega_\e| \le B\e$. 
\end{remark}

Since our a-priori estimates bound only the $L^2$-norm of the recovery variable $s$, the natural convergence for this component is going to be the weak convergence in $L^2$. Therefore, it is handy to define a tailored convergence for this purpose.
\begin{definition} 
	We say that a sequence of discrete triples $\mathbf{U}^\e = (U_i^\e, U_e^\e, S^\e)\in \mathbf{X}^\e$ converges to $\mathbf{U} = (U_i, U_e, S)\in L^2(\Omega)^3$ in the $s$-$s$-$w$ topology, denoted by $\mathbf{U}^\e \xrightarrow{s-s-w} \mathbf{U}$, if, as $\e \to 0$:
	\begin{align*}
		\Pi_\e U_i^\e &\to U_i \quad \text{strongly in } L^2(\Omega), \\
		\Pi_\e U_e^\e &\to U_e \quad \text{strongly in } L^2(\Omega), \\
		\Pi_\e S^\e &\rightharpoonup S \quad \text{ weakly in } L^2(\Omega).
	\end{align*}
\end{definition}

Since the limit functional $\Psi_n^\tau$ is defined on $L^2(\Omega)$ (see \eqref{MM-eps-tt}) and the approximating functionals $\Psi_n^{\e,\tau}$ (see \eqref{MM-eps}) are defined on discrete functions on $\V_\e$, as a first step we need to define a family of functionals $\widetilde{\Psi}_n^{\e,\tau}$ which coincide with $\Psi_n^{\e,\tau}$ on piecewise constant functions on $\Omega_\e$. Let $\widetilde{\Psi}_n^{\e,\tau}:  L^2(\Omega)^3 \to [0,+\infty]$ be defined as
\begin{equation}
	\label{def:psitilde}
	\widetilde{\Psi}_n^{\e,\tau}(\mathbf{U}):=
	\begin{cases}
		\Psi_n^{\e,\tau}(\mathbf{u}) &\text{if } \mathbf{U}=\Pi_\e \mathbf{u},\ \mathbf{u}\in \mathbf{X}^\e, \\
		+\infty &\text{otherwise}.
	\end{cases}
\end{equation}
The $\Gamma$-convergence of continuum limits of discrete interactions was studied in \cite{AC04} in a much more general setting. In Theorem \ref{th:ac04} and Lemma \ref{lemma:compactness} below we recall the main compactness and convergence results, adapted to the case of nearest-neighbour, scalar, quadratic interactions (precisely, see \cite[Theorem 3.1 and Remark 3.2]{AC04}) in the Dirichlet energy part of $\widetilde{\Psi}_n^{\e,\tau}$. In order to separately study this term, we define the functional $\mathcal{H}_\e : L^2(\Omega) \times L^2(\Omega) \to [0,+\infty]$ by
\begin{equation}
	\label{def:Hepsilon}
	\mathcal{H}_\e(U_i,U_e):=
	\begin{cases}
		a_i^\e(u_i)+a_e^\e(u_e) &\text{if } U_i=\Pi_\e u_i,\ U_e=\Pi_\e u_e, \\
		+\infty &\text{otherwise,}
	\end{cases}
\end{equation}
where
\[
a_{i,e}^\e(u) := \e^d\!\!\!\sum_{\{x_j,x_k\}\in \mathcal{E}_\e} G_{i,e}^\e(x_j,x_k)\left(\frac{u(x_j)-u(x_k)}{\e}\right)^2,\quad \text{for }u\in X^\e. 
\]

\begin{theorem}
	\label{th:ac04}
	For every sequence $(\e_j)$ of positive real numbers converging to 0, there exists a subsequence $(\e_{j_k})$ and two Carath\'eodory functions $g_{i,e}:\Omega \times \R^d \to \R$, satisfying
	\[
	\bar\alpha |y|^2 \leq g_{i,e}(x,y) \leq \bar\alpha^{-1} |y|^2,
	\]
	such that $\mathcal{H}_{\e_{j_k}}$ $\Gamma$-converges with respect to the $L^2(\Omega)$-topology to the functional $\mathcal{H}:L^2(\Omega)\times L^2(\Omega) \to [0,+\infty]$ defined as
	\begin{equation}
		\label{def:H}
		\mathcal{H}(u_i,u_e):=
		\begin{cases}
			\displaystyle \int_\Omega g_i(x,\nabla u_i(x)) +g_e(x,\nabla u_e(x))\ d x &\text{if } (u_i,u_e)\in H^1(\Omega)^2 \\
			+\infty &\text{otherwise.}
		\end{cases}
	\end{equation}
	Moreover, $g_{i,e}(x,\cdot)$ are quadratic forms on $\R^d$: for all $x\in \Omega$ there exist matrices $M_{i,e}(x)\in \R_\text{sym}^{d \times d}$ such that 
	\begin{equation}
		\label{eq:defMie}
		g_{i,e}(x,y)=y^TM_{i,e}(x)y\quad \text{for all }y\in \R^d.
	\end{equation}
\end{theorem}

For $r\in\{1,\dots,d\}$, we define the set of interior nodes with a forward neighbour in the direction of the $r$-th canonical basis vector $\mathbf{e}_r \in \mathbb{R}^d$ as $\V_\e^{\mathrm{int}} :=\{x_k\in \V_\e:\ x_k+\e \mathbf{e}_r\in \V_\e\}$. The corresponding discrete forward gradient is then given by:
\[D_{\e,r}w(x_k):=\frac{w(x_k+\e \mathbf{e}_r)-w(x_k)}{\e}.\]
By \eqref{eq:coercivity_strong} and \eqref{def:Hepsilon}, there exists a constant $M>0$ such that for all $u_i,u_e\in X^\e\times X^\e$
\begin{equation}
	\label{eq:grad-control}
	\sum_{r=1}^d \left(\|\Pi_\e D_{\e,r} u_i\|_{L^2(\Omega)}^2 +\|\Pi_\e D_{\e,r} u_e\|_{L^2(\Omega)}^2 \right)\le M \mathcal{H}_\e(U_i,U_e).
\end{equation}
In the spirit of the discrete Rellich-Kondrachov theorem, these estimates ensure (\cite{AC04}) that the family of reconstructed functions remains relatively compact in $L^2(\Omega)$ as the lattice scale $\e \to 0$.

\begin{lemma}
	\label{lemma:compactness}
	Let $G_\e:\mathcal E_\e \to \R$ satisfy $G_\e(x_j,x_k)\geq \alpha>0$ for all $\{x_j,x_k\}\in \mathcal E_\e$, for all $\e>0$, and let $\{w_\e\} \subset X^\e $ be a sequence such that
	\begin{equation}
		\label{eq:bounded_H1}
		\sup_{\e>0} \left( |w_\e|_\e^2 + \e^d\hspace{-0.5cm}\sum_{\{x_j,x_k\}\in \mathcal E_\e} G_\e(x_j,x_k) \left(\frac{w_\e(x_j)-w_\e(x_k)}{\e}\right)^2\right) < +\infty.
	\end{equation}
	Then there exist a function $w \in H^1(\Omega)$ and a subsequence $\{\e_j\}_{j \in \mathbb{N}} \to 0$ such that $\Pi_{\e_j} w_{\e_j} \to w$ strongly in $L^2(\Omega)$. Moreover, if $w^{\e_j} \in X_0^{\e_j}$, then
	\(
	\int_\Omega w \, dx = 0.
	\)
\end{lemma}

We now consider the minimizers $\mathbf{U}^\e$ of  $\Psi_n^{\e,\tau}$, for a fixed iteration $n \in \{1, \dots, N\}$, time step $\tau =T/N> 0$ and $(n-1)$-step solution $\mathbf{U}_{n-1}^{\e,\tau}$. (For sake of notation, we drop the indexes $n$ and $\tau$ from $\mathbf{U}^\e$.) 
The stability of the Euler scheme \eqref{stability-estimate-knm} ensures that the boundedness of the energy is preserved for all $n \le N$, if the initial data energy is bounded. Consequently, we assume that the sequence of initial data $\{\mathbf{U}_0^{\e}\}_{\e > 0}$ is well-prepared:
\begin{equation}
	\label{AS}
	\sup_{\e>0} \left( \frac{1}{2} \mathbf{a}^\e\left(\mathbf{U}_0^{\e}\right) + \phi^\e\left(\mathbf{U}_0^{\e}\right) \right) < \infty.
\end{equation}

\begin{proposition}
	\label{proposition:compactness}
	Let $\tau > 0$ and $n \in \{0, \dots, N\}$ be fixed, $(\varepsilon_j)_{j \in \mathbb{N}}$ be a sequence of positive real numbers converging to $0$, and $\mathbf{U}^{\varepsilon_j}=(U_{i}^{\varepsilon_j}, U_{e}^{\varepsilon_j}, S^{\varepsilon_j}) \in \mathbf{X}^{\varepsilon_j}$ be the minimizer of $\Psi_n^{\varepsilon_j,\tau}$. Under assumption \eqref{AS}, there exists a limit function $\mathbf{U}=(U_{i}, U_{e}, S) \in H^1(\Omega) \times H^1(\Omega) \times L^2(\Omega)$ such that, up to a subsequence, not relabeled, 
	\[
	\Pi_{\varepsilon_j} \mathbf{U}^{\varepsilon_j} \xrightarrow{s-s-w} \mathbf U \quad \text{as } j \to \infty.
	\]
	Furthermore, since $U_{e}^{\varepsilon_j} \in X_0^{\varepsilon_j}$ for all $j$, the limit extracellular potential satisfies the zero-mean constraint $\displaystyle\int_\Omega U_{e} \, dx = 0$.
\end{proposition}

\begin{proof} 
	The stability estimates for the Euler method \eqref{stability-estimate-knm} provide a constant $C>0$, independent of the time-step $\tau>0$, such that, for all iterations $n$, $E^\e(\mathbf{U}^{\e}) \le C E^\e(\mathbf{U}_0^{\e})$. Since by \eqref{eq:ab_coercive} and \eqref{eq:energy}, for $\mathbf{U}^{\e} \in \mathbf{X}^\e$
	\[
	3E^\e(\mathbf{U}^{\e}) \geq \mathbf{a}^\e(\mathbf{U}^{\e})+ \dfrac{1}{2} \mathbf{a}^\e(\mathbf{U}^{\e}) 
	+ \phi^\e(\mathbf{U}^{\e}) \geq \mathbf{a}^\e(\mathbf{U}^{\e})+\mathbf{b}^\e(\mathbf{U}^{\e})\geq \nu \|\mathbf{U}^{\e}\|^2_{\e}, 
	\]
	%
	we can apply Lemma \ref{lemma:compactness} to $\{U_{i}^{\e}\}$ and $\{U_{e}^{\e}\}$ and find a subsequence $\e_j \to 0$ and functions $U_{i} ,U_{e} \in H^1(\Omega)$ such that
	\begin{equation*}
		\Pi_{\e_j} U_{i,e}^{\e_j} \to U_{i,e} \quad \text{strongly in } L^2(\Omega).
	\end{equation*}
	Regarding the recovery variable, estimate \eqref{eq:ab_coercive} directly implies that the sequence of interpolants $\{\Pi_\e S^{\e}\}_\e$ is uniformly bounded in $L^2(\Omega)$. We may then extract a further subsequence (still denoted by $\e_j$) and a limit function $S \in L^2(\Omega)$ such that
	\begin{equation*}
		\Pi_{\e_j} S^{\e_j} \rightharpoonup S \quad \text{weakly in } L^2(\Omega).
	\end{equation*}
	Finally, since $U_{e}^{\e} \in X_0^\e$, in the limit we obtain the zero-mean constraint for $U_e$.
\end{proof}

\subsection{$\Gamma$-convergence of the incremental functionals}
The subsequent analysis is devoted to establishing the $\Gamma$-convergence of the family of discrete functionals $\widetilde\Psi_n^{\e,\tau}$ (defined in \eqref{def:psitilde}) towards the continuous functional $\Psi_n^\tau$ (defined in \eqref{MM-eps-tt}), with respect to the mixed $s$-$s$-$w$ topology. To this end, we first report two auxiliary results concerning the asymptotic behaviour of the quadratic forms and the nonlinear potential terms. 

\begin{lemma}
	\label{lemma:Gamma_norm}
	The family of functionals
	\[
	F^\e: L^2(\Omega) \to [0,+\infty],\qquad F^\e(W):=	\begin{cases}
		\|\Pi_\e w\|^2_{L^2(\Omega)}, & \text{if } W = \Pi_\e w,\ w \in X^\e \\
		+\infty, & \text{otherwise},
	\end{cases}
	\]
	$\Gamma$-converges in the strong and in the weak $L^2(\Omega)$ topologies to $F: L^2(\Omega) \to \R$, $F(W):=\|W\|^2_{L^2(\Omega)}$.
\end{lemma}

\begin{proof} 
	Given a sequence $w_\e \in X^\e$ such that $\Pi_\e w_\e \rightharpoonup w$, weakly in $L^2(\Omega)$, the liminf inequality is a consequence of the weak lower semicontinuity of the $L^2$-norm. 
	In order to establish the limsup inequality, for any $W \in L^2(\Omega)$, we define the recovery sequence $\{w_\e\}$ by setting the nodal values $w_\e(x_k)$ via cell averages over the cubes $Q_\e(x_k)$ centered at the lattice nodes $x_k \in \V_\e$:
	\begin{equation*}
		w_\e(x_k) := \frac{1}{\e^d} \int_{Q_\e(x_k)} W(x) \, dx,\qquad W_\e:=\Pi_\e w_\e.
	\end{equation*}
	Since $|\Omega\setminus \Omega_\e|$ vanishes as $\e\to0$ (see Remark \ref{lemma:meanzero}), by the standard approximation of $L^2$ functions by cell averages, the piecewise constant interpolation satisfies $\Pi_\e w_\e \to W$ strongly in $L^2(\Omega)$ and therefore $F^\e(W_\e)=\|\Pi_\e w_\e\|_{L^2(\Omega}^2 \to \|W\|_{L^2(\Omega}^2$. The liminf inequality w.r.t. the strong convergence and the limsup inequality w.r.t. to the weak one are a consequence of the former two.
\end{proof}

\begin{lemma}
	\label{lemma:gamma_g}
	Let $\mathbf U^\e = (U_i^\e, U_e^\e, S^\e)$ and $\mathbf W^\e = (W_i^\e, W_e^\e, R^\e)$ belong to $\mathbf{X}^\e$. Assume that $\mathbf U^\e \xrightarrow{s-s-w} \mathbf u$ and $\mathbf W^\e \xrightarrow{s-s-w} \mathbf w$. 
	Then,
	\begin{equation*}
		\lim_{\e\to0} \mathbf{g}^\e(\mathbf U^\e, \mathbf W^\e) = \mathbf{g}(\mathbf u, \mathbf w).
	\end{equation*}
\end{lemma}
\begin{proof} 
	By definition of $\mathbf{g}^\e$ and $\mathbf g$,
	\begin{align*}
		\mathbf{g}^\e(\mathbf U^\e,\mathbf W^\e)
		&=
		\lambda_F\langle U_i^\e-U_e^\e,\,
		W_i^\e-W_e^\e\rangle_\e +\delta\langle S^\e,\,
		W_i^\e-W_e^\e\rangle_\e
		+\sigma\langle U_i^\e-U_e^\e,\,
		R^\e\rangle_\e \\
		&=\lambda_F\int_\Omega
		\Pi_\e(U_i^\e-U_e^\e)\,
		\Pi_\e(W_i^\e-W_e^\e)\,dx\\
		&+\delta\int_\Omega
		\Pi_\e (S^\e)\,
		\Pi_\e(W_i^\e-W_e^\e)\,dx  +\sigma\int_\Omega
		\Pi_\e(U_i^\e-U_e^\e)\,
		\Pi_\e (R^\e)\,dx \\
		& = \mathbf{g}\left(\Pi_\e\mathbf U^\e,\Pi_\e\mathbf W^\e \right).
	\end{align*}
	Since the first integrand is the product of two strongly converging sequences and the other two are products of a strongly converging and a weakly converging sequence, we can pass to the limit as $\e \to 0$ in the integrals and obtain the thesis. 
\end{proof}

\begin{theorem} 
	\label{theorem:gamma}
	Let the time step $\tau > 0$ and the time level $n \in \{1, \dots, N\}$ be fixed. Let $\e_j \to 0$ be a sequence such that $\mathcal H_{\e_j}$ $\Gamma$-converges to $\mathcal H$ as in Theorem \ref{th:ac04} and assume that the sequence of vectors $(\mathbf{U}_{n-1}^{\e_j,\tau}) \subset \mathbf{X}^{\e_j}$ converges to $\mathbf{U}_{n-1} \in L^2(\Omega)^3$ in the $s$-$s$-$w$ topology. If
	\begin{equation}
		\label{eq:ASn-1}
		\sup_{\e_j>0} \left(  \mathbf{a}^{\e_j}(\mathbf{U}_{n-1}^{\e_j,\tau}) + \phi^{\e_j}(\mathbf{U}_{n-1}^{\e_j,\tau}) \right) < \infty,
	\end{equation}
	then the following properties hold:
	\begin{enumerate}
		\item[(I)] Equi-coercivity and Compactness: Any sequence of discrete states $(\mathbf{U}^{\e_j}) \subset \mathbf{X}^{\e_j}$ with uniformly bounded energy $\sup_{\e_j > 0} \Psi_n^{\e_j, \tau}(\mathbf{U}^{\e_j}) < \infty$ is relatively compact with respect to the $s$-$s$-$w$ topology. Furthermore, any limit $\mathbf{U} = (U_i, U_e, S)$ belongs to $\mathbf{X}= H^1(\Omega) \times H_\circ^1(\Omega) \times L^2(\Omega)$. 
		\item[(II)] $\Gamma$-convergence: The sequence of functionals $\widetilde\Psi_n^{\e_j, \tau}$ $\Gamma$-converges to the functional $\Psi_n^\tau$ with respect to the $s$-$s$-$w$ topology. Specifically, the following holds:
		\begin{enumerate}
			\item liminf inequality: For every sequence $(\mathbf{U}^{\e_j})$ such that $\mathbf{U}^{\e_j} \xrightarrow{s-s-w} \mathbf{U}$,
			\begin{equation*}
				\liminf_{j \to +\infty} \widetilde\Psi_n^{\e_j, \tau}(\mathbf{U}^{\e_j}) \ge \Psi_n^\tau(\mathbf{U}).
			\end{equation*}
			\item limsup inequality: For every $\mathbf{U} \in \mathbf{X}$, there exists a recovery sequence $(\widetilde{\mathbf{U}}_n^{\e_j,\tau})$ converging $s$-$s$-$w$ to $\mathbf{U}$ such that
			\begin{equation*}
				\limsup_{j \to +\infty} \widetilde\Psi_n^{\e_j, \tau}(\widetilde{\mathbf{U}}_n^{\e_j,\tau}) \le \Psi_n^\tau(\mathbf{U}).
			\end{equation*}
		\end{enumerate}
	\end{enumerate}
\end{theorem}

\begin{proof} 
	To prove (I), let $\mathbf{U}^\e=(U_i^\e,U_e^\e,S^\e)\in \mathbf{X}^\e$ be a sequence satisfying the uniform energy bound 
	\begin{align*}
		\sup_{\e>0} \Psi_n^{\e,\tau}(\mathbf{U}^\e) 
		&= \sup_{\e>0} \left(\frac{1}{2\tau} \mathbf{b}^\e(\mathbf{U}^\e - \mathbf{U}_{n-1}^{\e,\tau}) 
		+ \frac{1}{2} \mathbf{a}^\e(\mathbf{U}^\e)\right.\\
		&\qquad \left.+ \mathbf{\phi}^\e(\mathbf{U}^\e) 
		- \mathbf{g}^\e(\mathbf{U}_{n-1}^{\e,\tau}, \mathbf{U}^\e)\right) < \infty.
	\end{align*}	 
	Since the dissipation term $\mathbf{b}^\e$ is non-negative and the linear perturbation $\mathbf{g}^\e$ is sub-quadratic with respect to the energy, this hypothesis implies that the energy $E^\e(\mathbf{U}^\e)= \frac{1}{2} \mathbf{a}^\e(\mathbf{U}^\e) + \phi^\e(\mathbf{U}^\e)$ remains uniformly bounded independently of $\e$.  Proposition \ref{proposition:compactness} and Theorem \ref{th:ac04} thus ensure the existence of a subsequence $(\e_j)_{j\in \mathbb{N}} \to 0$ and a limit state $\mathbf{U} = (U_{i}, U_{e}, S) \in \mathbf{X}$ such that $\mathbf{U}^{\e_j} \xrightarrow{s-s-w} \mathbf{U}$ and $\mathcal H_{\e_j}$ $\Gamma$-converges to $\mathcal H$. For sake of notation, in the rest of the proof, we still denote this subsequence by $(\e)$ instead of $(\e_j)$.
	
	Regarding statement (II), we first establish the liminf inequality. Let $\mathbf{U}^\e \xrightarrow{s-s-w} \mathbf{U}$. We analyze the terms of $\Psi_n^{\e, \tau}$ as follows:
	\begin{itemize}
		\item \textit{Diffusive and Recovery parts:} By Theorem \ref{th:ac04}, and Lemma \ref{lemma:Gamma_norm}, 
		\begin{equation*}
			\liminf_{\e \to 0} \frac{1}{2} \mathbf{a}^\e(\mathbf{U}^\e) \ge \frac{1}{2} \mathbf{a}(\mathbf{U}).
		\end{equation*}
		\item \textit{Dissipative and Linear parts:} Since $\mathbf{U}^\e \xrightarrow{s-s-w} \mathbf{U}$ and $\mathbf{U}_{n-1}^{\e,\tau} \xrightarrow{s-s-w} \mathbf{U}_{n-1}$, Lemmas \ref{lemma:Gamma_norm} and \ref{lemma:gamma_g} yield
		\begin{align*}
			\liminf_{\e \to 0} \frac{1}{2\tau} \mathbf{b}^\e(\mathbf{U}^\e - \mathbf{U}_{n-1}^{\e,\tau}) &\ge \frac{1}{2\tau} \mathbf{b}(\mathbf{U} - \mathbf{U}_{n-1}), \\
			\lim_{\e \to 0} \mathbf{g}^\e(\mathbf{U}^\e, \mathbf{U}_{n-1}^{\e,\tau}) &= \mathbf{g}(\mathbf{U}, \mathbf{U}_{n-1}).
		\end{align*}
		\item \textit{Nonlinear potential:} 
		By the strong $L^2$-convergence of $U_{i}^{\e}-U_{e}^{\e}$, the lower semicontinuity of $\varphi$ and the standard lower semicontinuity of integral functionals, we obtain: 
		\begin{equation*}
			\liminf_{\e \to 0} \phi^\e(\Pi_\e\mathbf{U}^\e) = \liminf_{\e \to 0} \int_\Omega \varphi(\Pi_\e(U_{i}^{\e}-U_{e}^{\e}))d x \ge \int_\Omega \varphi(U_{i} - U_{e}) \, dx.
		\end{equation*}
	\end{itemize}
	Summing these contributions gives the required liminf inequality. To complete the proof of the $\Gamma$-convergence (II), it remains to establish the limsup inequality. We invoke \cite[Proposition 6.5]{PSCF05} for the set 
	\[
	D:= \left\{ (u_i,u_e,s)\in H^1(\Omega)\times H^1_\circ(\Omega)\times L^2(\Omega) : u_i-u_e \in L^\infty(\Omega)\right\},
	\]
	which satisfies a suitable density condition thanks to \cite[Lemma 4.19]{PSCF05}.
	Given a limit state $\mathbf{U}=(U_i, U_e, S) \in D$, owing to Theorem \ref{th:ac04}, it is possible to construct a recovery sequence $\widetilde{\mathbf{U}}^{\e} = (\widetilde{U}_{i}^{\e}, \widetilde{U}_{e}^{\e}, \widetilde{S}^{\e}) \in \mathbf{X}^\e$ such that $\widetilde{\mathbf{U}}^{\e} \xrightarrow{s-s-w} \mathbf{U}$ as $\e \to 0$,
	\begin{equation}
		\label{eq:linfty}
		\lim_{\e \to 0}\mathcal{H}_\e(\widetilde{U}_{i}^{\e}, \widetilde{U}_{e}^{\e}) = \mathcal{H}(U_i,U_e),\quad \text{and}\quad \sup \| \Pi_\e (\widetilde{U}_{i}^{\e}- \widetilde{U}_{e}^{\e})\|_{L^{\infty}(\Omega)} <+\infty.
	\end{equation}
	Moreover, $\widetilde{S}^{\e}$ defined as the cell-average interpolation of the limit function $S$, satisfies $\Pi_\e \widetilde{S}^\e \to S$ strongly in $L^2(\Omega)$ and thus, by \eqref{eq:linfty} and Lemma \ref{lemma:Gamma_norm}
	\begin{equation*}
		\lim_{\e \to 0}  \mathbf{a}^\e(\widetilde{\mathbf{U}}^{\e}) = \lim_{\e \to 0} \mathcal{H}_\e(\widetilde{U}_{i}^{\e}, \widetilde{U}_{e}^{\e}) + \gamma|\widetilde{S}^\e|_\e^2= \mathcal{H}(U_{i}, U_{e}) + \gamma\|S\|_{L^2(\Omega)}^2=   \mathbf{a}(\mathbf{U}).
	\end{equation*}
	By \eqref{eq:linfty} and continuity of $\varphi$, the terms $v_\e:=\Pi_\e (\widetilde{U}_{i}^{\e}- \widetilde{U}_{e}^{\e})$ and $-\varphi(v_\e)$ are uniformly bounded in $L^\infty$ and arguing as in the liminf estimate we obtain
	\begin{equation*}
		\liminf_{\e \to 0} \int_\Omega -\varphi(v_\e(x))dx \geq \int_\Omega -\varphi(U_i - U_e) dx
	\end{equation*}
	and thus
	\[
	\limsup_{\e \to 0} \int_\Omega \varphi(v_\e(x))dx \leq \int_\Omega \varphi(U_i - U_e) dx.
	\]
	Finally, we consider the metric and linear terms. The state $\mathbf{U}_{n-1}^{\e,\tau}$ is fixed and converges to $\mathbf{U}_{n-1}$ in the $s$-$s$-$w$ topology. Since the recovery sequence $\widetilde{\mathbf{U}}^{\e}$ converges strongly in $L^2(\Omega)$ in all three components, the quadratic dissipation $\mathbf{b}^\e$ and the linear perturbation $\mathbf{g}^\e$ satisfy, by Lemmas \ref{lemma:Gamma_norm} and \ref{lemma:gamma_g}:
	\begin{align*}
		\lim_{\e \to 0} \frac{1}{2\tau} \mathbf{b}^\e(\widetilde{\mathbf{U}}^{\e} - \mathbf{U}_{n-1}^{\e,\tau}) &= \frac{1}{2\tau} \mathbf{b}(\mathbf{U} - \mathbf{U}_{n-1}), \\
		\lim_{\e \to 0} \mathbf{g}^\e(\widetilde{\mathbf{U}}^{\e}, \mathbf{U}_{n-1}^{\e,\tau}) &= \mathbf{g}(\mathbf{U}, \mathbf{U}_{n-1}).
	\end{align*}
\end{proof}

The combination of the $\Gamma$-convergence result with the equi-coercivity of the incremental functional $\widetilde\Psi_n^{\e,\tau}$ allows us to characterize the asymptotic behaviour of the discrete solutions. This ensures that the sequence of discrete minimizers remains within a compact set and converges to a minimizer of the scaled problem \cite[Corollary 7.20]{DM93}. The uniqueness of the minimizer of the time-discrete Bidomain problem and the separate lower-semicontinuity of the terms $\mathbf{a}^\e, \mathbf{b}^\e, \phi^\e$ yield the following result:
\begin{corollary}
	Let $\{\mathbf{U}_n^{\e, \tau}\}_{\e>0}$ be a sequence of minimizers of the discrete functional $\widetilde\Psi_n^{\e, \tau}$, assume that $\mathcal H = \Gamma-\lim_{\e\to 0}\mathcal H_\e$  and let $\mathbf{U}_n^\tau \in H^1(\Omega)\times H^1_\circ(\Omega)\times L^2(\Omega)$ be the unique minimizer of the scaled functional $\Psi_n^\tau$. If  $\{\mathbf{U}_{n-1}^{\e,\tau}\}_{\e > 0}  \xrightarrow{s-s-w} \mathbf{U}_{n-1}$ and \eqref{eq:ASn-1} holds, then:
	\begin{enumerate}
		\item[(i)] $\mathbf{U}_n^{\e, \tau} \xrightarrow{s-s-w} \mathbf{U}_n^\tau$ as $\e \to 0$.
	\item[(ii)] The discrete minimum energies converge to the continuous minimum energy:
	\begin{equation*}
		\lim_{\e \to 0} \widetilde\Psi_n^{\e, \tau}(\mathbf{U}_n^{\e, \tau}) = \Psi_n^\tau(\mathbf{U}_n^\tau).
	\end{equation*}
	\item[(iii)] $\displaystyle \lim_{\e \to 0} \mathbf{a}^{\e}(\mathbf{U}_n^{\e, \tau}) = \mathbf{a}(\mathbf{U}_n^\tau),\quad \lim_{\e \to 0} \mathbf{b}^{\e}(\mathbf{U}_n^{\e, \tau}) = \mathbf{b}(\mathbf{U}_n^\tau),\quad \lim_{\e \to 0} \phi^{\e}(\mathbf{U}_n^{\e, \tau}) = \phi(\mathbf{U}_n^\tau)$.
\end{enumerate}
\end{corollary}
As a consequence, we also obtain the following result:
\begin{corollary}
\label{cor:minimizers}
Let $v^\e \in X^\e$ and $v\in L^2(\Omega)$ be such that
\(
\lim_{\e \to 0}|v^\e|_\e = \|v\|_{L^2(\Omega)}.
\) 
If $\bar u^\e = (u_i^\e,u_e^\e)$ is the unique solution of the minimum problem
\[
\min \left\{ \mathcal{H}^\e(\bar u): \bar u \in X^\e \times X_0^\e,\ u_i-u_e = v^\e \right\} 
\]
and $\limsup_{\e\to 0} \mathcal{H}^\e(\bar u^\e)<+\infty$, then
\[
\Pi_\e \bar u^\e \to \bar u\quad \text{strongly in }L^2(\Omega)^2\quad 
\text{and}\quad  \lim_{\e \to 0}\mathcal{H}^\e(\bar u^\e)=\mathcal{H}(\bar u),
\]
where $\bar u$ is the unique solution of the minimum problem
\[
\min \left\{ \mathcal{H}(\bar u): \bar u \in H^1(\Omega) \times H^1_\circ(\Omega),\ u_i-u_e = v \right\}.
\]
and $\mathcal{H}^\e,\mathcal{H}$ were defined in \eqref{def:Hepsilon} and \eqref{def:H}.
\end{corollary}
We now have all the elements to complete the proof of the main result.

\begin{proof}[Proof of Theorem \ref{th:main}]
Let $\mathbf{u}^{\e_j} = (u_i^{\e_j},u_e^{\e_j},s^{\e_j})$ be a sequence of  solutions of the KNM$-\e_j$ problem, such that $\mathcal H_{\e_j}$ $\Gamma$-converges to $\mathcal H$. Let $M_i$ and $M_e$ be the conductivity tensors in $\mathcal H$ and let $\mathbf{u}=(u_i,u_e,s)$ be the solution of the corresponding BD model. Fix $t \in [0, T]$. Owing to estimates \eqref{apriori-estimate} and Corollary \ref{cor:minimizers}, we only have to show that $\Pi_{\varepsilon_j} v^{\varepsilon_j}(t)=\Pi_{\varepsilon_j}\left(u_i^{\varepsilon_j}(t) - u_e^{\varepsilon_j}(t)\right) \to v(t)$ and that $\Pi_{\varepsilon_j} s^{\varepsilon_j}(t) \to s(t)$, strongly in $L^2(\Omega)$. For any given time step size $\tau > 0$, we denote by $(V^{\varepsilon_j,\tau},S^{\varepsilon_j,\tau})$ and $(V^\tau,S^\tau)$ the time-dependent piecewise linear interpolants of the discrete solutions generated by the minimizing movements schemes \eqref{MM-eps} and \eqref{MM-eps-tt}, respectively. Dropping the variable $t$ for sake of notation, we have:
\begin{equation}
	\label{eq:error_decomp}
	\begin{aligned}
		& \|\Pi_{\varepsilon_j}(v^{\varepsilon_j},s^{\varepsilon_j}) - (v,s)\|_{L^2(\Omega)^2}\leq \, \| \Pi_{\varepsilon_j}(v^{\varepsilon_j},s^{\varepsilon_j}) - \Pi_{\varepsilon_j}(V^{\varepsilon_j,\tau},S^{\varepsilon_j,\tau})\|_{L^2(\Omega)^2} \\
		&\quad \quad + \|\Pi_{\varepsilon_j}(V^{\varepsilon_j,\tau},S^{\varepsilon_j,\tau}) - (V^{\tau},S^{\tau})\|_{L^2(\Omega)^2} 
		+ \|(V^{\tau},S^{\tau}) - (v,s)\|_{L^2(\Omega)^2}.
	\end{aligned}
\end{equation}
We begin by estimating the first and third terms on the right-hand side of \eqref{eq:error_decomp}. By hypothesis, the sequence of well-prepared microscopic initial data satisfies the uniform energy bound $\limsup_{j \to \infty} E^{\varepsilon_j}(\mathbf{u}_0^{\varepsilon_j}) = M < \infty$. Consequently, by \eqref{final-error-estimate}, we have
\begin{equation}
	\label{eq:c1tau}
	\begin{aligned}
		\| \Pi_{\varepsilon_j}(v^{\varepsilon_j},s^{\varepsilon_j}) - \Pi_{\varepsilon_j}(V^{\varepsilon_j,\tau},S^{\varepsilon_j,\tau})\|_{L^2(\Omega)^2}^2  &= | v^{\varepsilon_j} -V^{\varepsilon_j,\tau}|_{\varepsilon_j}^2 +| s^{\varepsilon_j}-S^{\varepsilon_j,\tau}|_{\varepsilon_j}^2, \\
		&\leq \max\left\{\frac{1}{C_m},1\right\}\mathbf{b}^{\varepsilon_j}(\mathbf{u}^{\varepsilon_j}-{\mathbf{U}}^{\varepsilon_j, \tau}),\\
		& \leq C_1^2 \tau,
	\end{aligned}
\end{equation}
where the positive constant $C_1$ depends exclusively on $M$, $C_m$, and $C$. Analogously, applying the corresponding continuous time-discretization error estimate \eqref{final-error-estimate-bidomain} we obtain
\begin{equation}
	\label{eq:c2tau}
	\|(V^{\tau},S^{\tau}) - (v,s)\|_{L^2(\Omega)^2}
	\leq C_2 \sqrt\tau.
\end{equation}
Next, we address the spatial limit $j \to \infty$ for the remaining central term of \eqref{eq:error_decomp} at the fixed time-step level $\tau > 0$. Let $n \in \{1, \dots, N\}$ denote the specific time interval containing the instant $t$. We carry out an inductive argument on the time step index $n$. For $n=0$, the strong convergence holds by the hypothesis on the initial data. Assume now that $\Pi_{\varepsilon_j} \mathbf{U}_{n-1}^{\varepsilon_j,\tau} \xrightarrow{s-s-w} \mathbf{U}_{n-1}^\tau$ as $j \to \infty$. Under this inductive step, Corollary 1 guarantees that the sequence of unique discrete minimizers $\mathbf{U}_n^{\varepsilon_j,\tau} = (U_{i,n}^{\varepsilon_j,\tau}, U_{e,n}^{\varepsilon_j,\tau}, S_n^{\varepsilon_j,\tau})$ satisfies
\begin{equation*}
	\Pi_{\varepsilon_j} \mathbf{U}_n^{\varepsilon_j,\tau} \xrightarrow{s-s-w} \mathbf{U}_n^\tau = (U_{i,n}^\tau, U_{e,n}^\tau, S_n^\tau) \quad \text{as } j \to \infty,
\end{equation*}
where $\mathbf{U}_n^\tau$ is the unique minimizer of the continuous functional $\Psi_n^\tau$. To upgrade the weak convergence $\Pi_{\varepsilon_j} S_n^{\varepsilon_j,\tau} \rightharpoonup S_n^\tau$ in $L^2(\Omega)$ to the strong topology, we exploit the convergence of $\mathbf{b}^{\varepsilon_j}(\mathbf{U}_n^{\varepsilon_j,\tau})$ guaranteed by Corollary 1 (iii), which implies that $\lim_{j \to \infty} \|\Pi_{\varepsilon_j} S^{\varepsilon_j,\tau}_n\|_{L^2(\Omega)} = \|S_n^\tau\|_{L^2(\Omega)}.$ 
Since the linear interpolant $\mathbf{U}^{\varepsilon_j, \tau}(t)$ is a convex combination of the discrete vectors $\mathbf{U}_{n-1}^{\varepsilon_j,\tau}$ and $\mathbf{U}_n^{\varepsilon_j,\tau}$ on each time subinterval, this strong convergence naturally extends pointwise in time to the instant $t$, yielding
\begin{equation}
	\label{eq:pescayogurt}
	\lim_{j \to \infty}\|\Pi_{\varepsilon_j}(V^{\varepsilon_j,\tau}(t),S^{\varepsilon_j,\tau}(t)) - (V^{\tau}(t),S^{\tau}(t))\|_{L^2(\Omega)^2} = 0.
\end{equation}
Passing to the limit superior as $j \to \infty$ on both sides of inequality \eqref{eq:error_decomp}, by \eqref{eq:c1tau}, \eqref{eq:c2tau}, and \eqref{eq:pescayogurt}, we obtain
\begin{equation*}
	\limsup_{j \to \infty}  \| \Pi_{\varepsilon_j} (v^{\varepsilon_j}(t),s^{\varepsilon_j}(t)) - (v(t),s(t))\|_{L^2(\Omega)^2}  \leq (C_1 + C_2) \sqrt{\tau}.
\end{equation*}
Taking the limit as the time step size $\tau \to 0^+$ concludes the proof.
\end{proof}


\end{document}